\documentclass[10pt,reqno]{amsart}
\numberwithin{equation}{section}
\usepackage{amsmath,amsthm}
\usepackage{mathptmx}  
\usepackage{amssymb}
\usepackage{multicol}
\DeclareFontFamily{U}{BOONDOX-calo}{\skewchar\font=45 }
\DeclareFontShape{U}{BOONDOX-calo}{m}{n}{
  <-> s*[1.05] BOONDOX-r-calo}{}
\DeclareFontShape{U}{BOONDOX-calo}{b}{n}{
  <-> s*[1.05] BOONDOX-b-calo}{}
\DeclareMathAlphabet{\mathcalboondox}{U}{BOONDOX-calo}{m}{n}
\SetMathAlphabet{\mathcalboondox}{bold}{U}{BOONDOX-calo}{b}{n}
\DeclareMathAlphabet{\mathbcalboondox}{U}{BOONDOX-calo}{b}{n}

\newcommand{\mcb}[1]{{\mathcalboondox #1}}
\usepackage{amssymb}
\usepackage{amsmath}
\usepackage{amsthm}
\usepackage{diagbox}
\usepackage{booktabs}
\usepackage{enumitem}
\usepackage[sans]{dsfont} 
\usepackage{bbm}
\usepackage{changebar}
\usepackage[colorlinks]{hyperref}
\usepackage{a4wide}
\usepackage[usenames,dvipsnames,svgnames,table]{xcolor}
\usepackage{xcolor}

\usepackage{tikz}
\usetikzlibrary{arrows,shapes,automata,petri,positioning,calc}
\tikzset{
    place/.style={
        circle,
        thick,
        draw=black,
        fill=gray!50,
        minimum size=20mm,
    },
        state/.style={
        circle,
        thick,
        draw=blue!75,
        fill=blue!20,
        minimum size=20mm,
    },
}

\tikzset{
    cross/.pic = {
    \draw[rotate = 45] (-0.2,0) -- (0.2,0);
    \draw[rotate = 45] (0,-0.2) -- (0, 0.2);
    }
}

\usepackage{ulem}
\usepackage{setspace}

\newtheorem{thm}{Theorem}[section]
\newtheorem{lem}[thm]{Lemma}
\newtheorem{cor}[thm]{Corollary}

\newtheorem{prop}[thm]{Proposition}

\newtheorem{definition}[thm]{Definition}

\newtheorem{rem}[thm]{Remark}

\newcommand\ve{\varepsilon}

\newcommand{\bb}[1]{{\mathbb #1}}

\title[{Hydrodynamic behavior of long-range symmetric exclusion with a slow barrier}
]{Hydrodynamic behavior of  long-range symmetric\\ exclusion with a slow barrier: diffusive regime}
\author{Pedro Cardoso, Patr\'icia   Gon\c calves, Byron Jim\'enez-Oviedo}
\newcommand{\Addresses}{{
		\footnotesize
		Pedro Cardoso, \textsc{\noindent Center for Mathematical Analysis,  Geometry and Dynamical Systems \\
			Instituto Superior T\'ecnico, Universidade de Lisboa\\
			Av. Rovisco Pais, no. 1, 1049-001 Lisboa, Portugal}\par\nopagebreak
		\textit{E-mail address}: \texttt{pedro.gondim@tecnico.ulisboa.pt}
		
		\medskip
		
		Patr\'icia   Gon\c calves, \textsc{\noindent Center for Mathematical Analysis,  Geometry and Dynamical Systems \\
Instituto Superior T\'ecnico, Universidade de Lisboa\\
Av. Rovisco Pais, no. 1, 1049-001 Lisboa, Portugal}\par\nopagebreak
		\textit{E-mail address}: \texttt{pgoncalves@tecnico.ulisboa.pt}
		
		\medskip
		
		Byron Jim\'enez-Oviedo, \textsc{\noindent Escuela de Matem\'atica,  \\
Faculdad de Ciencias Exactas y Naturales, Universidad Nacional de Costa Rica\\
Heredia, Costa Rica}\par\nopagebreak
		\textit{E-mail address}: \texttt{byron.jimenez.oviedo@una.cr}
}}
\begin{document}
\begin{abstract}
    In this article we analyse the hydrodynamical behavior of the  symmetric exclusion process
with long jumps and in the presence of a slow barrier. The jump rates for fast bonds are  given by a transition probability $p(\cdot)$ which is symmetric and has finite variance, while for slow bonds the jump rates are given $p(\cdot)\alpha n^{-\beta}$ (with $\alpha>0$ and $\beta\geq 0$), and correspond to  jumps from $\mathbb{Z}_{-}^{*}$ to $\mathbb N$. We prove that:  if there is a fast bond from $\mathbb{Z}_{-}^{*}$ and $\mathbb N$, then the hydrodynamic limit is given by the heat equation with no boundary conditions; otherwise, it is given by the previous equation if $0\leq \beta<1$, but for $\beta\geq 1$ boundary conditions appear, namely, we get Robin (linear) boundary conditions if $\beta=1$ and Neumann boundary conditions if $\beta>1$.
\end{abstract}
\maketitle{}
\section{Introduction}
Over the  last years, there has been an intensive research activity around the derivation of  the  hydrodynamic limits of the conserved quantities of an interacting particle system \cite{kipnis1998scaling}. 
 For systems with only one conservation law, the space-time evolution of the conserved quantity is ruled by a   partial differential equation (PDE), namely the \textit{hydrodynamic equation}.  When the system has more than one conservation law, the hydrodynamic limit is given by a system of   equations that can be coupled. 
 When the system has long range interactions, a crossover from a diffusive behavior to a super-diffusive behavior can be expected, and depending on the tail behavior of the transition probability. 
For systems which are superposed with different types of dynamics, one expects that either the nature of the PDE changes or boundary conditions might appear.

Our focus on this article  is to analyse the propagation of the local perturbation of  the dynamical microscopic rules at the level of the macroscopic hydrodynamic equations.  More specifically, we consider as toy model, the exclusion process with long jumps introduced in \cite{jara2009hydrodynamic} but evolving in the one-dimensional lattice $\mathbb Z$.  In this process, particles evolve as 
one-dimensional continuous time random walks, with the exclusion constraint stating that two particles cannot occupy the same site at any given time.  At every  bond $\{x,y\}$, with $x,y\in\mathbb Z$, we attach a Poisson process with parameter $p(y-x)$, where \textcolor{black}{$p:\mathbb Z\to\mathbb R$} is a symmetric transition probability given explicitly by \eqref{eq:trans_prob}.

 In \cite{jara2009hydrodynamic} the hydrodynamic limits for the exclusion process and for the zero-range process with long jumps, both evolving on $\mathbb Z^d$, were analysed.  There it was analysed the case where $p(\cdot)$ is given by \textcolor{black}{$\|z\|^{-(\gamma+d)}$} and for $\gamma$ restricted to the range $\gamma\in(0,2)$ for which a super-difusive behavior is expected. By taking the anomalous time scale $n^\gamma$, the   hydrodynamic equation was given by the fractional heat equation.  In \cite{sunder} it was considered  a general class  of misanthrope interacting particle systems on $\mathbb Z^d$, which  included both the exclusion process and the zero-range process.  In this general class of models the  dynamics  conserved the number  of particles.    The transition probability that was considered is  as above but  asymmetric, i.e. for $z\in\mathbb Z^d$, \textcolor{black}{ $p(z)=\mathbbm{1}_{\{z>0\}}\|z\|^{-(\gamma +d)}$}, where $z>0$ meant that $z=(z_1,\cdots, z_n)$ and $z_i\geq 0$ and $z\neq 0$.   In this case, two types of evolution equations were obtained. For $\gamma\in(0,1)$ the hydrodynamics  was given by a integro-partial differential equation, when the system  was taken on the anomalous time scale $n^\gamma$; and  for $\gamma\geq 1$, the hydrodynamics  was given by the  Burgers  equation,  when the time scale  was the Euler scaling $n$.   This last behavior  was the same as in the finite-range setting.   In the critical case, corresponding to $\gamma=1$, by correcting the time scale by a factor $ \log(n)$ the hydrodynamic behavior   was shown to be the same as when $\gamma>1$,  i.e. the Burgers equation. 

In this article,  we restrict our study to the one-dimensional setting and to the exclusion process with a symmetric transition probability as in \cite{jara2009hydrodynamic}.   Moreover,   we perturb the exclusion dynamics in such a way that whenever a jump occurs from negative to positive sites, the rate is slowed down  with respect to the rates in all the other bonds. The goal is to create a \textit{slow} barrier and to see its macroscopic effect at the nature of the boundary conditions of the PDE.  This is reminiscent of the work in  \cite{FGNAIHP,franco2015phase} where  the same problem was considered but for particles performing nearest-neighbor  jumps and with a slower rate whenever a particle crosses the bond  $\{-1,0\}$.    Here we  combine these two properties: particles can give  jumps that are arbitrarily large and the slow rates are attached to all the  bonds connecting $\mathbb{Z}_{-}^{*}=\{-1,-2,\ldots \}$ to $\mathbb{N}=\{0,1,2,\ldots,\}$.  These rates depend on two parameters  $\alpha >0$ and $\beta \geq 0$. This means that particles jump from $x$ to $y$ with probability $p(y-x)$ but every time a jump occurs from $\mathbb{Z}_{-}^{*}$ to $\mathbb N$ the jump rate becomes  equal to $\alpha n^{-\beta}p(y-x)$, see Figure \ref{fig1}.
 In this article we   analyse the case when $p(\cdot)$ has finite variance and therefore, the  system has a diffusive behavior,  i.e. the hydrodynamic equation is the heat equation.  In a companion article \cite{CGJ2} we analyse the case when $p(\cdot)$ has infinite variance, for which the behavior is super-diffusive.

\begin{figure}[htb!]
	
	\begin{center}
		\begin{tikzpicture}[thick]

		\draw[shift={(-5.01,-0.15)}, color=black] (0pt,0pt) -- (0pt,0pt) node[below]{\dots};
		\draw[shift={(5.01,-0.15)}, color=black] (0pt,0pt) -- (0pt,0pt) node[below]{\dots};

		\draw[-latex] (-5.3,0) -- (5.3,0) ;
		\draw (1.3,-1.5) -- (1.3,2);
		\draw[latex-] (-5.3,0) -- (5.3,0) ;
		\foreach \x in {-4.5,-4,-3.5,...,4.5}
		\pgfmathsetmacro\result{\x*2-3}
		\draw[shift={(\x,0)},color=black] (0pt,0pt) -- (0pt,-2pt) node[below]{\scriptsize \pgfmathprintnumber{\result}};
		
		
		\node[ball color=black, shape=circle, minimum size=0.3cm] (C) at (1.5,0.15) {};
		
		\node[ball color=black, shape=circle, minimum size=0.3cm] (D) at (2.5,0.15) {};
		
		\node[ball color=black, shape=circle, minimum size=0.3cm] (E) at (-4,0.15) {};
		
		\node[draw=none] (S) at (3.5,0.15) {};
		\node[draw=none] (R) at (-0.5,0.15) {};
		\node[draw=none] (L) at (-4,0.15) {};
		\node[draw=none] (M) at (-1,0.15) {};

		\path [<-] (S) edge[bend right =70, color=blue]node[above] {\footnotesize $\dfrac{1}{2}p(2)$}(D);
		\path [->] (C) edge[bend right =70, color=blue]node[above] {\footnotesize $\dfrac{\alpha}{2n^{\beta}} p(-4)$}(R);			
		\path [<-] (M) edge[bend right =70, color=blue]node[above] {\footnotesize $\dfrac{1}{2}p(6)$}(L);

		\end{tikzpicture}
		\caption{Example of the dynamics of the model. The  jumps rates of size $z$ inside $\mathbb{Z}_{-}^{*}$ and $\mathbb{N}$ are equal to $\frac{1}{2}p(z)$, but  between $\mathbb{Z}_{-}^{*}$ and $\mathbb{N}$ they are equal to $\frac{\alpha}{2n^{\beta}}p(z)$.}	
		\label{fig1}
	\end{center}	
\end{figure}
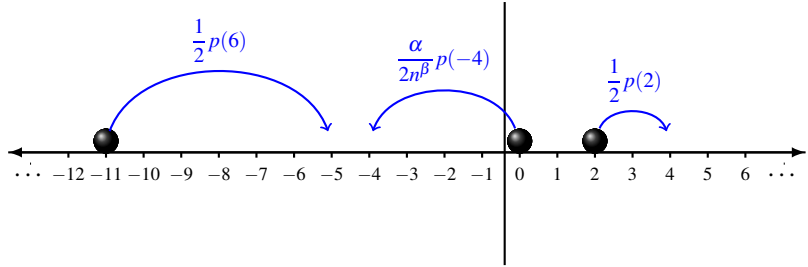

 We observe that if for every $x,y \in \mathbb{Z}$,  there exists at least one path through which a particle can move between $x$ and $y$ (without using slow  bonds), then we have only one hydrodynamic equation, regardless of the values of $\alpha$ and $\beta$. This is a consequence of the fact that the number of slow bonds is not sufficiently big in order to create a different behavior from the case where slow bonds are not present.   On the other hand, if that path does not exist, then we have the same phase transition  that has been studied in \cite{FGNAIHP,franco2015phase}, where the value of $\alpha$ is only relevant in the critical case corresponding to $\beta=1$. We distinguish two cases which are completely equivalent in terms of the hydrodynamic limit: the finite-range case, which is a model very similar to the one in \cite{franco2015phase}, and the long-range case, where the transition probability is \textcolor{black}{a slight modification of the one in \cite{jara2009hydrodynamic}} \textcolor{black}{, since now we choose $\gamma>2$ in order to produce a diffusive behavior}.  The results of this article can be summarized in the following way: a) if particles can move between $\mathbb{Z}_{-}^{*}$ to $\mathbb N$ using a path without slow bonds then we obtain the heat equation as hydrodynamic equation; b) if particles  cannot move between $\mathbb{Z}_{-}^{*}$ to $\mathbb N$ without using at least one slow bond, then the hydrodynamic equation is the heat equation with no boundary conditions if $0\leq \beta<1$, with Robin (linear) boundary conditions if $\beta=1$ and Neumann boundary conditions if $\beta>1$.

 Other similar models have also been recently studied, as  the symmetric long range exclusion in contact with slow/fast reservoirs, which has been explored in a series of articles \cite{bernardinARMA,bernardin2017slow,bernardin2021hydrodynamic,goncalvesscotta} and is a extension of the nearest-neighbor case studied in \cite{baldasso}.  In \cite{bernardin2017slow} it was analysed the exclusion process with $p(\cdot)$   as in \eqref{eq:trans_prob} and with finite variance and the system is put in contact with slow/fast reservoirs.  In  this case, the hydrodynamic equation is the heat equation with Dirichlet, Robin or Neumann boundary  conditions; or a reaction-diffusion  equation; or simply a reaction equation, both with Dirichlet boundary conditions, depending on the strength of the reservoirs' dynamics.   
In the other articles, the case where $p(\cdot)$ has infinite variance has been studied, leading to a fractional PDE given in terms of a regional fractional Laplacian.

In \cite{tertumariana} the analogue of our model in case of nearest-neighbour jumps was  analysed in the $d$-dimensional torus and the resulting hydrodynamic equation was the heat equation with Neumann, Robin or no boundary conditions. We believe that our results could be easily extended to the $d$-dimensional setting  by combining our results with the results of \cite{tertumariana}, but we leave this to a future work.

 We would also like to comment on the main difficulties that we  have encountered along the proofs.  Due to the presence of slow bonds, our natural space of test functions  presents discontinuities both in the Robin and Neumann regimes. Since we work with a long jumps' model which is evolving in the infinite volume, one has to carefully identify the terms that will contribute to the limiting equations. This can be done by properly truncating some series that appear from the action of the  generator, and analysing their tails, and then closing carefully the remaining series in terms of the density of particles. However, the  replacement lemmas that we need to derive in order to close the equations and recognise the corresponding boundary conditions are not straightforward since we consider  all the range $\beta\geq 0$. Moreover, we also  have had to prove the uniqueness of the weak solutions that we derived. To that end,  energy estimates are needed in order to show two properties of the profiles: that they live in a  Sobolev space and  that their translations with respect to constant profiles,  live in the space $L^2(\mathbb R)$. This is quite different from the finite volume case where usually the latter condition  is immediate. Ultimately, in the proof of uniqueness, it  has been challenging to obtain density arguments to approximate the test functions which present discontinuities, since we are working with functions defined on the full line.

\textbf{Outline of the article:} In Section \ref{sec:model_hydro} we define the model, the notions of weak solution to the PDEs that we  have obtained and the hydrodynamic limit. In Section \ref{sectight} we prove the tightness of the sequence of measures induced by the empirical measure and  in Section \ref{seccharac}  we characterize the limiting point. In Section \ref{secenerest} we prove the energy estimates and in Section \ref{secheurwithout} we prove several estimates that are needed to characterize uniquely the limiting point. 
\section{The model and hydrodynamics}\label{sec:model_hydro}
\subsection{The exclusion process with slow bonds in $\mathbb{Z}$}
\label{sec:model}

We begin by establishing the notation for some sets $\mathbb{Z}:=\{\ldots, -1, 0, 1, \ldots \}$, $\mathbb{Z}_{-}^{*}:=\{-1,-2,\ldots \}$ and $\mathbb{N}:=\{0,1,2,\ldots,\}$. Moreover, we denote $\mathbb{R}_{-}^{*}:= (- \infty, 0)$, $\mathbb{R}_{+}^{*}:= (0, \infty)$ , $\mathbb{R}_{+}:=[0, \infty)$ and $\mathbb{R}^{*} = \mathbb{R}_{-}^{*} \cup \mathbb{R}_{+}^{*}=\mathbb{R}-\{0\}$. 

Now we will describe the dynamics of our process. The transitions occur at  all the bonds $\{x,y \in \mathbb{Z}: x \neq y\}$. Here, we identify the bonds $\{x,y\}$ and $\{y, x\}$ (for instance $\{1,2\}=\{2,1\}$). We will denote the set of bonds by $\mcb B$. The elements of the lattice are called \textit{sites} and are denoted by Latin letters such as $x,y,z$. 

The state space of our Markov process is $\Omega:=\{0,1\}^{\mathbb{Z}}$. The elements of $\Omega$ are called configurations and are denoted by Greek letters such as $\eta, \xi$. Given  a configuration $\eta$ and a site $x$, we denote the number of particles at $x$ by $\eta(x)$.
Given a bond $\{x,y\}$, a particle can only move between $x$ and $y$ if $\eta(x) \neq \eta(y)$ and in this case, $\eta(x)$ and $\eta(y)$ exchange their values and  produce a new configuration $\eta^{x,y} \in \Omega$, which can be defined as:   
\begin{equation*}
\eta^{x,y}(z) = \eta(y)\mathbbm{1}_{z=x}+\eta(x)\mathbbm{1}_{z=y}+ \eta(z)\mathbbm{1}_{z\neq x,x}
\end{equation*}
A particle will move across  a bound  $\{x,y\}$ with probability $p(x-y)$, where $p: \mathbb{Z} \rightarrow \mathbb{R}$ is a transition probability  with four properties: it allows jumps of size $1$ ($p(1)=p(-1) >0$), forbids jumps of size $0$ ($p(0)=0$),  it is symmetric ($p(z) = p(-z), \forall z \in \mathbb{Z} $) and it has finite variance ($\sum_z z^2 p(z) < \infty$). We observe that we impose that jumps of size $1$ are possible, just for a matter of taste, since in this case the results of \cite{franco2015phase} are a consequence of our general results. All the following results hold for any transition probability distribution satisfying these four assumptions. Therefore, the reader can either assume we deal with a finite range model or a long-range model. In the finite-range model, we assume that there exists a range $k \in \mathbb{N}$ and $p_1, p_2, \ldots, p_k \in [0,1]$ such that $\sum_{j=1}^{k} p_j = \frac{1}{2}$. Then $p(j)=p_{|j|}$ if $1 \leq |j| \leq k$ and $p(j)=0$ otherwise.
The long-range model allows jumps of arbitrary size according to the transition probability $p(\cdot)$ given by
\begin{align}\label{eq:trans_prob}
p(z) =
\begin{cases}
0, \; \; \text{if} \; \; z=0, \\
c_{\gamma}|z|^{-\gamma-1}  , \; \; \text{if} \; \; z \neq 0,
\end{cases}
\end{align}
where $ \gamma>2$ (in order to produce a finite variance) and $c_{\gamma}$ is a normalizing constant. These two models are completely equivalent from our viewpoint. We will denote
\begin{align*}
m: = \sum_{z \in \mathbb{N}} z p(z) \quad \textrm{and} \quad \sigma^2: = \sum_z z^2 p(z) < \infty.
\end{align*}
Now we will consider a set of slow bonds 
\begin{equation*}
    \mcb S \subset \mcb S_0:=\big\{\{x,y\} \in \mcb B: x <0, y \geq 0 \big\}.
\end{equation*}

The complement of $\mcb S$ with respect to $\mcb B$ will be denoted by $\mcb F$, the set of  fast bonds. Let $n$ be a positive integer, $\alpha >0, \beta \geq 0$. Given a configuration $\eta$, we will denote the rate of transitions in a bond $\{x,y\}$  by $\xi_{x,y}^n(\eta)$, which is defined by 
\textcolor{black}{
\begin{align*}
\xi_{x,y}^n(\eta) =
\begin{cases}
\alpha n^{-\beta}  [\eta(x)(1-\eta(y))+\eta(y)(1-\eta(x))], \{x,y\} \in \mcb S, \\
 [\eta(x)(1-\eta(y))+\eta(y)(1-\eta(x))], \{x,y\} \in \mcb F.
 \end{cases}
\end{align*} 
}

Observe that bonds in $\mcb S$  coined the name \textit{slow bonds} as a consequence of  the definition of the rates and the fact that $\beta\geq 0$.
We say that a function $f: \Omega \rightarrow \mathbb{R}$ is {local} if there exists a finite $\Lambda \subset \mathbb{Z}$ such that $f$ is determined by $ \{ \eta(x): x \in \Lambda\}$. This means that if $\eta_1, \eta_2 \in \Omega$ are such that $\eta_1 (x) = \eta_2(x), \forall x \in \Lambda$, then $f(\eta_1)=f(\eta_2)$. Our Markov process is described by its infinitesimal generator $\mcb {L}_{n}$, which is defined on local functions $f:\Omega\to\mathbb{R}$ by 
\begin{align*}
\mcb{L}_{n}f(\eta) :=& \frac{1}{2}  \sum_{\{x, y\} \in \mcb B} p(x-y) \xi^{n}_{x,y}(\eta)[f(\eta^{x,y})-f(\eta)].
\end{align*}
For every $a \in (0,1)$, we define the product measure $\nu_{a}$ on $\Omega$, with marginals given by $\nu_{a}\{\eta \in \Omega: \eta(x)=1\} = a, \forall x \in \mathbb{Z}$. Under this measure, the random variables $\{ \eta(x): x \in \mathbb{Z} \}$ are independent and identically distributed with Bernoulli distribution of parameter  $a$. Since $p(\cdot)$ is symmetric, a simple computation shows that the measure $\nu_a$ is reversible with respect to $\mcb{L}_{n}$.

\textcolor{black}{Hereinafter, we denote $\sum_{x,y \in \mathbb{Z}: \{x,y \} \in \mcb S }$ by $\sum_{\{x,y \} \in \mcb S }$; this means that every slow bond $\{x_1, x_2\}$ is counted \textit{twice} in $\sum_{\{x,y \} \in \mcb S }$. We define $\sum_{\{x,z \} \in \mcb F }$ and $\sum_{\{x,z \} \in \mcb S_0 }$ in a similar way.} We will say that we have a slow barrier blocking the movement between $\mathbb{Z}_{-}^{*}$ and $\mathbb{N}$ if $\sigma_{\mcb S}^2=\sigma^2$, where
\begin{equation} \label{defsigmas}
\sigma_{\mcb S}^2 := \sum_{ \{x,y  \} \in \mcb S} |y-x| p(y-x) \leq \sum_{ \{x,y  \} \in \mcb {S}_0} |y-x| p(y-x) = \sigma^2.
\end{equation}

\subsection{Notation}

Now we will present some notation for functions which depend only on the space variable. For an interval $I$ in $\mathbb{R}$ and a number \textcolor{black}{$r \geq 1 \in  \mathbb{N}$}, we denote by $C^{r} \left(  I \right)$ the set of functions defined on $I$ that are $r$ times differentiable. Moreover, \textcolor{black}{$C^0(I)$ is the space of continuous functions $G: I \rightarrow \mathbb{R}$ and} $C^{\infty} (I):= \cap_{r=1}^{\infty} C^{r}(I)$. We also consider the set  $C_{c}^{r} ( \mathbb{R} )$ of functions $G \in C^{r}\left( \mathbb{R} \right)$ such that $G$ has a compact support that may include $0$.

Hereafter we fix $T>0$ and a finite time horizon $[0,T]$. Moreover, we denote $\eta_{t}^{n}(x):= \eta_{t n^2}(x)$, so that $\eta_t^n$ has  infinitesimal generator $n^2 \mcb{L}_{n}$. We observe that given an initial configuration $\eta_0^n \in \Omega$, the evolution of the Markov process $\{\eta_{t}^{n}; t \geq 0\}$ is a trajectory in $\Omega$ (i.e., $\eta_{t}^{n} \in \Omega, \forall t \in [0,T])$.  We define $\mcb D([0,T],\Omega)$ as the space of c\`adl\`ag trajectories (right-continuous and with left limits everywhere) $f:[0,T] \rightarrow \Omega$ with the Skorohod topology. In particular, $(\eta_{t}^{n})_{0 \leq t \leq T} \in \mcb D([0,T],\Omega)$.

Let $\mcb {M}^+(\mathbb{R})$ be the space of non-negative, Radon measures on $\mathbb{R}$, equipped with the weak topology. For   $\eta \in \Omega$, we define the empirical measure $\pi^{n}(\eta,du)$ by 
\begin{equation*}\label{MedEmp}
\pi^{n}(\eta, du):=\dfrac{1}{n}\sum _{x }\eta(x)\delta_{\frac{x}{n}}\left( du\right) \in \mcb{M}^+(\mathbb{R}),
 \end{equation*}
where $\delta_{b}$ is a Dirac measure on $b \in \mathbb{R}$. For  $G: \mathbb{R} \rightarrow \mathbb{R}$,  $\langle \pi^n, G \rangle$ denotes the integral of $G$ with respect to $\pi^n(\eta,du)$.

Let $\mcb g: \mathbb{R} \rightarrow [0,1]$ be a measurable function. We will assume that we have a sequence $(\mu_n)_{n \geq 1}$ of probability measures on $\Omega$ which are {associated to the profile $\mcb g$}. This means that for every function $G \in C_c^0(\mathbb{R})$ and for every $\delta >0$, it holds
\begin{equation*}
\lim_{n \rightarrow \infty} \mu_n \left( \eta \in \Omega: \Big| \langle \pi^n, G \rangle - \int_{\mathbb{R}} G(u) \mcb g(u) du \Big| > \delta \right) =0. 
\end{equation*}
For every $n \geq 1$, let $\mathbb{P} _{\mu_{n}}$ be the probability measure on $\mcb D([0,T],\Omega)$ induced by the Markov process $\{\eta_{t}^{n};{t\geq 0}\}$ and by the initial configuration $\eta_0^n$ with distribution $\mu_{n}$. We denote the expectation with respect to $\mathbb{P}_{\mu_{n}}$ by $\mathbb{E}_{\mu_{n}}$. We also define $\pi^{n}_{t}(\eta, dq):=\pi^{n}(\eta^{n}_{t}, dq) $. We observe that $(\pi_{t}^{n})_{0 \leq t \leq T}$ is a trajectory in $\mcb M^+(\mathbb{R})$ and a Markov process. We denote by  $\mcb D([0,T], \mcb{M}^+(\mathbb{R}))$ the space of c\`adl\`ag  trajectories $f:[0,T] \rightarrow \mcb{M}^+(\mathbb{R})$ with the Skorohod topology. In particular, $(\pi_{t}^{n})_{0 \leq t \leq T} \in \mcb D([0,T],\mcb M^+(\mathbb{R}))$. Finally, we define $(\mathbb{Q}_{n})_{n \geq 1}$ as the sequence of probability measures on $\mcb D([0,T],\mcb M^+(\mathbb{R}))$ induced by the Markov process $(\pi_{t}^{n})_{0 \leq t \leq T}$ and by the initial configuration $\eta_0^n$ with distribution $\mu_{n}$.

When we have a slow barrier effect ($\sigma_S^2 = \sigma^2$) separating $\mathbb{Z}_{-}^{*}$ and $\mathbb{N}$ at the microscopic level, we expect there will be some conditions in which we can see a macroscopic blockage of mass between $\mathbb{R}_{-}^{*}$ and $\mathbb{R}_{+}$. Because of this, in some cases it will be convenient to deal with functions which \textit{may be} discontinuous at the origin and have smooth restrictions in $\mathbb{R}_{-}$ and $\mathbb{R}_{+}$. We say that $G \in C_c^{\infty}(\mathbb{R}^{*})$ if there exist $G_{-},G_{+} \in C_c^{\infty} (\mathbb{R})$ such that $G(u)=G_{-}(u)\mathbbm{1}_{u <0}+G_{+}(u)\mathbbm{1}_{ u \geq 0}$.

\begin{rem} \label{Gcont}
Observe that $C_c^{\infty}(\mathbb{R}) \subset C_c^{\infty}(\mathbb{R}^{*})$ since  for any $G_{0} \in   C_c^{\infty}(\mathbb{R}) $, one can choose $G_{-}=G_{+}=G_0$.
\end{rem}

Regardless of the measure space $X$, we will always denote the Lebesgue measure in $X$ by $\mu$. In this way, $L^2( X):=L^2 (X, d \mu)$ is the space of functions $f: X \rightarrow \mathbb{R}$ such that $\int_X |f|^2 d \mu < \infty$ and its norm is denoted by $\| \cdot \|_{2,X}$. Also, $L^{\infty}( X):=L^{\infty} (X, d \mu)$ is the space of functions $f: X \rightarrow \mathbb{R}$ with finite essential supremum and its norm is denoted by $\| \cdot \|_{\infty}$. 

Following Section 8.2 of \cite{brezis2010functional}, we will define some Sobolev spaces. Given an open interval $I$, the Sobolev space $ \mcb{H}^1(I)$ is the set of functions $f \in L^2(I)$ such that there exists $g \in L^2(I)$
\begin{align*}
\int_{I} f(u) \tfrac{d \phi}{du} (u) du = - \int_{I} g(u) \phi(u) du, \forall \phi \in C_c^{\infty}(I). 
\end{align*}
Above, $g$ will be denoted by $\tfrac{d f}{d u}$ and it is the \textit{weak derivative} of $f$. 
The next result is a consequence from Proposition 8.1 of \cite{brezis2010functional}.
\begin{prop}
The space $\big(\mcb{H}^1(I),\|  \cdot \|_{\mcb{H}^1(I)} \big) $ is a separable Hilbert space, where  $\| \cdot \|_{\mcb{H}^1(I)}$ is defined by
\begin{align*}
\| f \|^2_{\mcb{H}^1(I)} := \| f \|_{2,I}^2 + \Big \| \frac{d f}{d u}  \Big \|^2_{2,I} .
\end{align*} 
\end{prop}

With an abuse of notation, we will say that $f \in \mcb{H}^{1} \left( \mathbb{R}^* \right)$ if  $f_{-}:=f|_{ \mathbb{R}_{-}^{*}} \in  \mcb{H}^{1} \left( \mathbb{R}_{-}^{*} \right)$ and $f_{+}:= f|_{ \mathbb{R}_{+}^{*}} \in  \mcb{H}^{1} \left( \mathbb{R}_+^{*} \right)$. Then Proposition \ref{repcont} tells us that $f_{-}$ and $f_{+}$ have continuous representatives $\tilde{f}_{-}$ and $\tilde{f}_{+}$ in $(-\infty,0]$ and $[0, \infty)$, respectively. If there exists $a \in \mathbb{R}$ such that $g:=f-a \in \mcb{H}^{1} \left( \mathbb{R}^* \right)$, we denote 
\begin{align*}
f(0^{+}) := \lim_{\varepsilon \rightarrow 0^{+}} \frac{1}{\varepsilon} \int_0^{\varepsilon} f(u) du =a+  \lim_{\varepsilon \rightarrow 0^{+}} \frac{1}{\varepsilon} \int_0^{\varepsilon} \tilde{g}_{+}(u) du =a+  \tilde{g}_{+}(0) ; \\
  f(0^{-}) := \lim_{\varepsilon \rightarrow 0^{+}} \frac{1}{\varepsilon} \int_{-\varepsilon}^{0} f(u) du =a+ \lim_{\varepsilon \rightarrow 0^{+}} \frac{1}{\varepsilon} \int_{-\varepsilon}^{0} \tilde{g}_{-}(u) du =a+ \tilde{g}_{-}(0). 
\end{align*}
Above, $\tilde{g}_{-}$ and $\tilde{g}_{+}$ are the continuous representatives of $g_{-}:=g|_{ \mathbb{R}_{-}^{*}}$ and $g_{+}:= g|_{ \mathbb{R}_{+}^{*}}$ in $(-\infty,0]$ and $[0, \infty)$, respectively.

\begin{rem} \label{Hcont}
 Observe that $\mcb{H}^1(\mathbb{R}) \subset \mcb{H}^{1} \left( \mathbb{R}^* \right)$ since $f|_{ \mathbb{R}_{-}^{*}} \in  \mcb{H}^{1} \left( \mathbb{R}_{-}^{*} \right)$ and $f|_{ \mathbb{R}_{+}^{*}} \in  \mcb{H}^{1} \left( \mathbb{R}_+^{*} \right)$ for every $f \in \mcb{H}^1(\mathbb{R})$.
\end{rem}
In the particular case where $f \in \mcb{H}^1(\mathbb{R})$, by uniqueness of the continuous representative, $\tilde{f}$ coincides with $\tilde{f}_{-}$ and $\tilde{f}_{+}$ on $(-\infty,0]$ and $[0, \infty)$, respectively, and we get
$
f(0^{+}) = f_{+}(0) = f(0) = f_{-}(0) =  f(0^{-}).
$
However, in the more general case $f \in \mcb{H}^{1} \left( \mathbb{R}^* \right)$, since there is not a global representative in $\mathbb{R}$, we cannot assume that $f(0^{+})=f(0^{-})$.

In a very close way to Definition 23.1 of \cite{zeidler1989nonlinear}, we say that $\varrho: [0,T ] \rightarrow L^2(I)$ is in the space $L^2 \big(0, T ; \mcb{H}^1 (I ) \big)$ if $\varrho(t,\cdot) \in \mcb{H}^1 ( I )$ for almost every $t$ on $[0,T]$ and
\begin{align*}
\|\varrho\|^2_{L^2 \big(0, T ; \mcb{H}^1 ( I ) \big)} :=  \int_0^T \|\varrho(t, \cdot) \|^2_{\mcb{H}^1 \left( I \right)} dt < \infty.
\end{align*}
Given $\varrho \in L^2 \big(0, T ; \mcb{H}^1 ( I) \big)$, we say that $\tilde{\varrho}$ is the continuous representative of $\varrho$ if $\tilde{\varrho}(t,\cdot) \in \mcb{H}^1 ( I )$ and $\tilde{\varrho}(t,\cdot) \in C^0( \bar{I} )$ for almost every $t$ on $[0,T]$. From Proposition \ref{repcont}, we have $\varrho = \tilde{\varrho}$ almost everywhere on $[0,T] \times \mathbb{R}$ and $\tilde{\varrho}$ is unique. With an abuse of notation, we will say that $\varrho \in L^2 \left(0, T ; \mcb{H}^1 \left( \mathbb{R}^{*} \right) \right) $ if $\varrho(t,\cdot) \in \mcb{H}^1 (\mathbb{R}^{*} )$ for almost every $t$ on $[0,T]$. From Remark \ref{Hcont}, we have $L^2 \left(0, T ; \mcb{H}^1 \left( \mathbb{R} \right) \right) \subset L^2 \left(0, T ; \mcb{H}^1 \left( \mathbb{R}^{*} \right) \right)$.

Still following \cite{zeidler1989nonlinear}, we will need to define some spaces for the test functions. Given a \textcolor{black}{normed vector space} $(N, \| \cdot \|_{N} )$, we say that $P ( [0,T],  N )$ is the space of all polynomials $G: [0,T] \rightarrow N$, i.e., there exists $k \in \mathbb{N}$ such that $G(t)=a_0 + a_1 t + \ldots + a_k t^k$, with $a_j \in N, \forall j=0,1,\ldots,k$, $\forall t \in [0,T]$. In our context, we will have $(N, \| \cdot \|_{N} )=(C_c^{\infty}(\mathbb{R}),\| \cdot \|_{\mcb{H}^1(\mathbb{R})})$ or $(N, \| \cdot \|_{N} )=(\mcb{H}^1(\mathbb{R}), \| \cdot \|_{\mcb{H}^1(\mathbb{R})})$ and the norm in  $P ( [0,T],  N )$ will be $\| \cdot \|_{L^2 \big(0, T ; \mcb{H}^1 ( \mathbb{R} ) \big)}$.
With an abuse of notation, we will say that $G \in P \big( [0,T],  C_c^{\infty}(\mathbb{R}^{*}) \big)$ if there exists $k \in \mathbb{N}$ such that $G(t)=a_0 + a_1 t + \ldots + a_k t^k$, with $a_j \in C_c^{\infty}(\mathbb{R}^{*}), \forall j=0,1,\ldots,k$, $\forall t \in [0,T]$. 
 It is easy to see that for every $G \in P \big( [0,T],  C_c^{\infty}(\mathbb{R}^{*}) \big)$, there exist $G_{-},G_{+} \in P \big( [0,T],  C_c^{\infty}(\mathbb{R}) \big)$ such that
\begin{align*}
G (t,u) =\mathbbm{1}_{u\in(-\infty,0)}G_{-} (t,u)+\mathbbm{1}_{u\in [0, \infty)} G_{+} (t,u).
\end{align*}
With the same reasoning as in Remark \ref{Gcont}, we conclude that $ P \big( [0,T],  C_c^{\infty}(\mathbb{R}) \big) \subset P \big( [0,T],  C_c^{\infty}(\mathbb{R}^*) \big)$.
Given $G \in  P \big( [0,T],  C_c^{\infty}(\mathbb{R} ^{*}) \big)$, there exists $b>0: G(s,u)=0$ when $|u|  \geq b$, for every $s \in [0,T]$. We denote
\begin{align}\label{eq:B_g}
b_G:= \min \big \{b \in \mathbb{N}: G(s,u)=0 \quad\textrm{for}\quad (s,u) \in [0,T] \times \big(  ( - \infty, -b ] \cup [b, \infty) \big) \big \}.
\end{align}
Hereinafter we write $f(n) \lesssim g(n)$ if there exists a constant $C$ independent of $n$ such that $f(n) \leq C g(n)$ for every $n \geq 1$.   Before enunciating the hydrodynamic limit for our model, we will present all the hydrodynamic equations. 

\subsection{Hydrodynamic Equations}
Now we define the notions of weak solution of the hydrodynamic equations that we obtain. 
\begin{definition}\label{eq:dif}
Let $\mcb g: \mathbb{R} \rightarrow [0,1]$ be a measurable function. We say that $\varrho :[0,T] \times \mathbb{R} \rightarrow [0,1]$ is a weak solution of the heat equation in $\mathbb{R}$ with initial condition $g$ 
\begin{equation} \label{eqhyddifreal}
\begin{cases}
\partial_t \varrho (t,u) = \frac{\sigma^2}{2} \Delta \varrho (t,u), (t,u) \in [0,T] \times \mathbb{R}, \\
\varrho(0,u) =\mcb g(u), u \in \mathbb{R} 
\end{cases}
\end{equation}
if the following two conditions hold:
\begin{enumerate}
\item
for every $t \in [0,T]$, for every $G \in \mcb S_{\textrm {Dif}}:=P \big( [0,T] , C_c^{\infty}(\mathbb{R}) \big)$, we have  $F_{\textrm{Dif}}(t, \varrho,G, \mcb g)=0$, where
\begin{align*}
F_{\textrm{Dif}}(t, \varrho,G, \mcb g):= & \int_{\mathbb{R}} \varrho(t,u) G(t,u) du - \int_{\mathbb{R}} \mcb g(u) G(0,u) du 
-  \int_0^t \int_{\mathbb{R}} \varrho(s,u) \Big[ \dfrac{\sigma^2}{2} \Delta + \partial_s \Big] G(s,u) du ds;
\end{align*}
\item 
there exists $a \in (0,1)$ such that $\bar{ \varrho} \in L^2 \left(0, T ; \mcb{H}^1 ( \mathbb{R} ) \right)$, where $\bar{ \varrho}:=\varrho-a$.
\end{enumerate}
\end{definition}
\textcolor{black}{
\begin{rem}
We note that in the second condition of last definition we need only that $\bar{ \varrho}(s, \cdot) \in \mcb{H}^1 ( \mathbb{R} )$ for almost every $s \in [0,T]$; we do not require this to be satisfied for a particular value of $s$ (e.g. for $s=0$), hence we do not assume any additional integrability property of the initial profile. 
\end{rem}
}

\begin{definition}
Let $\kappa\geq 0$ and $\mcb g: \mathbb{R} \rightarrow [0,1]$ be a measurable function. We say that $\varrho:[0,T] \times \mathbb{R} \rightarrow [0,1]$ is a weak solution of the heat equation in $\mathbb{R}^{*}$ with Robin boundary conditions and initial condition $\mcb g$ 
\begin{equation} \label{eqhyddifrob}
\begin{cases}
\partial_t \varrho(t,u) = \frac{\sigma^2}{2} \Delta \varrho(t,u), (t,u) \in [0,T] \times \mathbb{R}, \\
\partial_u \varrho(t,0^+) = \partial_u \varrho(t,0^-)= \kappa[ \varrho(t,0^+) - \varrho(t,0^{-})], t \in (0,T], \\
\varrho(0,u) =\mcb  g(u), u \in \mathbb{R} 
\end{cases}
\end{equation}
if the following two conditions hold:
\begin{enumerate}
\item
for every $t \in [0,T]$, for every $G \in \mcb S_{\textrm {Rob}}:=P \big( [0,T] , C_c^{\infty}(\mathbb{R}^{*}) \big)$, we have $F_{\textrm{Rob}}(t, \varrho,G,\mcb g,\kappa)=0$, where 
\begin{align*}
F_{\textrm{Rob}}(t, \varrho,G, g, \kappa):= & \int_{\mathbb{R}} \varrho(t,u) G(t,u) du - \int_{\mathbb{R}} \mcb g(u) G(0,u) du - \int_0^t \int_{\mathbb{R}} \varrho(s,u) \Big[ \dfrac{\sigma^2}{2} \Delta + \partial_s \Big] G(s,u) du ds   \\
+ &  \frac{\sigma^2}{2} \int_0^t [ \partial_u G(s,0^{-})  \varrho(s,0^{-})  - \partial_u G(s,0^{+})  \varrho(s,0^{+}) ] ds \\
+ & \frac{\kappa \sigma^2}{2} \int_0^t  [   \varrho(s,0^{+})  - \varrho(s,0^{-}) ]  [  G(s,0^{+})    -  G(s,0^{-})  ] ds;  
\end{align*}
\item 
there exists $a \in (0,1)$ such that $\bar{ \varrho} \in L^2 \left(0, T ; \mcb{H}^1 ( \mathbb{R}^* ) \right)$, where $\bar{ \varrho}:=\varrho-a$.
\end{enumerate}
\end{definition}

\begin{rem}
If in last definition we take $\kappa=0$, then we denote $F_{\textrm{Neu}}(t, \varrho,G, g)=F_{\textrm{Rob}}(t, \varrho,G, g,0)$ and we say that $\varrho$ is a weak solution to the heat equation with Neumann boundary conditions. 
\end{rem}

The uniqueness of weak solutions of \eqref{eqhyddifreal} and  \eqref{eqhyddifrob}  is proved in Appendix \ref{secuniq}.

\subsection{The main result}

Now we will enunciate the hydrodynamic limit of our model.
\begin{thm} \textbf{(Hydrodynamic Limit)} \label{hydlim}
Let $\mcb g: \mathbb{R} \rightarrow [0,1]$ be a measurable function. Let $(\mu_n)_{n \geq 1}$ be a sequence of probability measures in $\Omega$ associated to the profile $\mcb g$ such that $H( \mu_n | \nu_a) \lesssim  n$, for some $a \in (0,1)$. Then, for any $0 \leq t \leq T$, any $G \in C_c^0(\mathbb{R})$ and any $\delta >0$,
\begin{align*}
\lim_{n \rightarrow \infty} \mathbb{P}_{\mu_{n}} \Big( \eta^n_\cdot \in  D([0,T],\Omega) : \Big| \langle \pi^n_t, G \rangle - \int_{\mathbb{R}} G(u) \varrho(t,u) du \Big| > \delta \Big) =0, 
\end{align*}
where $\varrho$ is the unique weak solution of
\begin{align*}
\begin{cases}
\eqref{eqhyddifreal},  &\quad  \text{if} \; \sigma_{\mcb S}^2 < \sigma^2 \; \text{or}, \; \sigma_{\mcb S}^2 = \sigma^2 \text{and} \; 0 \leq \beta <1; \\
\eqref{eqhyddifrob} \quad \textrm{with} \quad  \kappa= \frac{2m \alpha}{\sigma^2}, & \quad  \text{if}  \; \sigma_{\mcb S}^2 = \sigma^2 \text{and} \;  \beta =1; \\
\eqref{eqhyddifrob} \quad \textrm{with} \quad  \kappa= 0, &  \quad \text{if}  \; \sigma_{\mcb S}^2 = \sigma^2 \text{and} \;  \beta  >1.
\end{cases}
\end{align*}
\end{thm}
We observe that we have a \textit{static} behavior when $\sigma_{\mcb S}^2 < \sigma^2$, since in this case from \eqref{defsigmas} we have  $\mcb S\subsetneq\mcb S_0$ and therefore,  there exists \textit{at least one fast bond} $\{x_1,x_2\}$ with $x_1 <0$ and $x_2 \geq 0$ through which mass will flow normally between $\mathbb{R}_{-}^{*}$ and $\mathbb{R}_{+}$, regardless of the value of $\beta$. On the other hand, when $\sigma_{\mcb S}^2 = \sigma^2$ we have a \textit{phase transition} depending on the value of $\beta$, analogous to Theorem 4.1 of \cite{franco2015phase}. For $\beta \in [0,1)$, the slow bonds do not produce any macroscopic effect; for $\beta > 1$, there is no transport of mass between $\mathbb{R}_{-}^{*}$ and $\mathbb{R}_{+}$; and finally, in the critical case $\beta=1$, we have the boundary conditions, which depend on the value of $\alpha$.  The proof of the theorem is presented in the following sections. From here on we fix $a\in(0,1)$ such that $H( \mu_n | \nu_a) \lesssim  n$ and this will be useful in Sections \ref{secenerest} and \ref{secheurwithout}. In Section \ref{sectight}, we prove that the sequence $(\mathbb{Q}_n)_{n \geq 1}$ is tight with respect to the Skorohod topology of $\mcb D \big( [0,T], \mcb{M}^+(\mathbb{R}) \big)$ and therefore it has at least a limit point $\mathbb{Q}$. Combining the results  of Section  \ref{seccharac} and Section \ref{secheurwithout}, we prove that $\mathbb{Q}$ is concentrated on trajectories that satisfy the first condition of weak solutions of the corresponding hydrodynamic equations. In Section \ref{secenerest}, we prove that the second condition is also satisfied. The necessary replacement lemmas are proved in Section \ref{secheurwithout} and the uniqueness of the hydrodynamic equations is explained in Appendix \ref{secuniq}. Finally we present some auxiliary results in Appendix \ref{secuseres}.
\section{Tightness} \label{sectight}
In this section, we will prove that the sequence $(\mathbb{Q}_{n})_{ n \geq 1 }$ is tight by using Proposition 4.1.6  of \cite{kipnis1998scaling}. Thus, we need to show that, for every $\varepsilon >0$,
\begin{equation} \label{T1dif}
\displaystyle \lim _{\delta \rightarrow 0^+} \limsup_{n \rightarrow\infty} \sup_{\tau  \in \mathcal{T}_{T},\bar\tau \leq \delta} {\mathbb{P}}_{\mu _{n}}\Big[\eta_{\cdot}^{n}\in \mcb D ( [0,T], \Omega) :\left\vert \langle\pi^{n}_{\tau+ \bar\tau},G\rangle-\langle\pi^{n}_{\tau},G\rangle\right\vert > \ve \Big]  =0, 
\end{equation}
for any function $G$ belonging to $C_c^2(\mathbb{R})$. Here $\mathcal{T}_T$ as the set of stopping times bounded by $T$ and we assume that all the stopping times are bounded by $T$, which means $\tau+\bar{\tau}$ should be read as $\min \{ \tau + \bar{\tau}, T \}$.
 From Dynkin's Formula , see Appendix 1 in \cite{kipnis1998scaling}, we have  $\forall t \geq 0$ that
\begin{equation} \label{defMnt}
\mcb M_{t}^{n}(G) = \langle \pi_{t}^{n},G\rangle - \langle \pi_{0}^{n},G\rangle  - \int_0^t n^2 \mcb {L}_n \langle \pi_{s}^{n},G\rangle ds,
\end{equation}
is a martingale.  We observe that \eqref{T1dif} is a direct consequence of  Markov's inequality together with Proposition \ref{tightcond1dif} and Proposition  \ref{tightcond2dif}. We start with the next result which will be important in what follows and can be derived from algebraic manipulations. 
\begin{prop} \label{gendif}
For any $G$, it holds   
\begin{align}
  \int_{0}^{t} n^2 \mcb L_{n}\langle \pi_{s}^{n},G(s, \cdot) \rangle ds   =&  \int_{0}^{t} \frac{1}{n} \sum _{z } n^2  \mcb {K}_n G \left(s, \tfrac{z}{n} \right)  \eta_{s}^{n}(z) ds \label{princdif}   \\
+&  \int_{0}^{t} \frac{n}{2} \Big( 1 - \frac{\alpha}{n^{\beta}} \Big)  \sum_{\{x,z \} \in \mcb S }  [G(s,\tfrac{x}{n}) - G(s,\tfrac{z}{n})] p(z-x)  [\eta_s^n(x)-\eta_s^n(z)] ds \label{extradif} \\
=&  \int_{0}^{t} \alpha n^{1-\beta}  \sum_{ \{x,z\} \in \mcb S}  [G(s, \tfrac{x}{n}) - G(s, \tfrac{z}{n}) ]  p(x-z)  \eta_s^n(z) ds \label{robterm} \\
+&  \int_{0}^{t}\sum_{ \{x,z\} \in \mcb F} n [G(s, \tfrac{x}{n}) - G(s, \tfrac{z}{n}) ]  p(x-z)  \eta_s^n(z) ds, \label{neuterm}
\end{align}
where  
\begin{equation}\label{op_Kn}
\mcb{K}_n G \left( \tfrac{x}{n} \right): = \sum_{y} \left[ G( \tfrac{y}{n}) -G( \tfrac{x}{n}) \right] p(y-x) = \sum_{z } \left[ G( \tfrac{z+x}{n}) -G( \tfrac{x}{n})\right] p(z) .
\end{equation}
\end{prop}
Observe that above we wrote the integral term in two different ways for convenience, we will  use  \eqref{princdif} and \eqref{extradif} when $G \in \mcb S_{\textit {Dif}}$\,; and \eqref{robterm} and \eqref{neuterm} when $G \in \mcb S_{\textit {Rob}}$. \textcolor{black}{Moreover, every fixed slow bond (resp. fixed fast bond) is counted \textit{twice} in  $\sum_{ \{x,z\} \in \mcb S}$ (resp. $\sum_{ \{x,z\} \in \mcb F}$).}  
\begin{prop} \label{tightcond1dif}
For $G \in C_c^2(\mathbb{R})$, it holds
\begin{align*}
\lim_{\delta \rightarrow 0^+} \limsup_{n \rightarrow \infty} \sup_{\tau \in \mathcal{T}_T, \bar{\tau} \leq \delta} \mathbb{E}_{\mu_n} \left[ \Big| \int_{\tau}^{\tau+ \bar{\tau}} n^2  \mcb L_{n}\langle \pi_{s}^{n},G\rangle ds \Big| \right] = 0.
\end{align*}
\end{prop}
\begin{proof}
From Proposition \ref{gendif}, a Taylor expansion of first order on $G$ and the fact that   $|\eta_{s}^{n}(x)| \leq 1, \forall x \in \mathbb{Z}, \forall s \in [0,T]$, we conclude that
\begin{align*}
 |n^2 \mcb L_{n}\langle \pi_{s}^{n},G\rangle |=  \frac{1}{n} \sum _{ z } \Big|  n^2 \mcb {K}_n G \left(\tfrac{z}{n} \right) \Big|  +\frac{ \sigma^2 (1+ \alpha) \|G'\|_{\infty}}{2}.
\end{align*}
Applying Proposition \ref{convdisc} the proof  ends. 
\end{proof}
\begin{prop} \label{tightcond2dif}
For $G \in \mcb S_{Dif}$ and $\beta \in [0,\infty)$ or $G \in \mcb S_{Rob}$ \textcolor{black}{, $\sigma_{\mcb S}^2 = \sigma^2 $} and $\beta \in [1,\infty)$, it holds
\begin{align*}
\lim_{\delta \rightarrow 0^+} \limsup_{n \rightarrow \infty} \sup_{\tau \in \mathcal{T}_T, \bar{\tau} \leq \delta} \mathbb{E}_{\mu_n} \left[ \left( \mcb M_{\tau}^{n}(G) -  \mcb M_{\tau+\bar{\tau}}^{n}(G) \right)^2 \right] = 0.
\end{align*}
\end{prop}
\begin{rem}
The statement of the previous result also includes functions $G$ which are time dependent and may be discontinuous at the origin. This general result  is not necessary here for the proof of tightness, nevertheless, in Section \ref{seccharac}, we will need that result, so we decided to state it here. 
\end{rem}
\begin{proof}
From Dynkin's formula the expectation in the statement of the theorem is equal to 
\begin{align*}
 \mathbb{E}_{\mu_n} \Big[ \int_{\tau}^{\tau+\bar{\tau}}\Gamma^n_s(G)  ds \Big],
\end{align*} where 
$\Gamma^n_s(G)=n^2 \left(\mcb L_{n}[\langle \pi_{s}^{n},G (s, \cdot) \rangle]^{2}-2\langle \pi_{s}^{n},G(s, \cdot) \rangle \mcb L_{n} \langle \pi_{s}^{n},G (s, \cdot) \rangle\right).$
Simple computations show that for $G\in C_c(\mathbb{R})$,  we have
\begin{align*}
\Gamma^n_s(G)
=&  \frac{1}{2}  \sum_{\{w,z \} \in \mcb F } \big[ G\left( \tfrac{w}{n}\right) - G\left( \tfrac{z}{n}\right) ]{^2}  p(w-z) [\eta(w)-\eta(z)]^2 \\
+&  \frac{\alpha}{2n^{\beta}} \sum_{\{w,z \} \in \mcb S } \big[ G\left( \tfrac{w}{n}\right) - G\left( \tfrac{z}{n}\right) ]^2  p(w-z) [\eta(w)-\eta(z)]^2.
\end{align*}
From the fact that  $ |\eta_s^n(z)|\leq 1, \; \forall z \in \mathbb{Z}$ and that $G \in \mcb S_{{Dif}}$ last display can be bounded from above by
\begin{align*}
 \frac{\alpha +1}{2} \sum_{w,z }     \big[G\left(\tfrac{z}{n}\right)-G\left(\tfrac{w}{n}\right)\big ]^2 p(z-w)
\end{align*}
and the proof ends by applying  Proposition \ref{tight2condaux}. 
Now, if instead we take  $G \in \mcb S_{{ Rob}}$\textcolor{black}{, $\sigma_{\mcb S}^2 = \sigma^2 $} and $\beta \in [1,\infty)$, we have 
\begin{align*}
|\Gamma^n_s(G)| 
 \leq&  \frac{1}{2} \sum_{w=0}^{\infty} \sum_{z=0}^{\infty}  \big[ G_{+} \left(s, \tfrac{w}{n}\right) - G_{+}\left(s,\tfrac{z}{n}\right) ]^2  p(w-z)  
 + \frac{1}{2} \sum_{w=-\infty}^{-1} \sum_{z=-\infty}^{-1}  \big[ G_{-} \left(s, \tfrac{w}{n}\right) - G_{-} \left(s,\tfrac{z}{n}\right) ]^2  p(w-z)  \\
+&  \frac{\alpha}{2n} \sum_{\{w,z \} \in \mcb S } \big[ G\left(s, \tfrac{w}{n}\right) - G\left(s, \tfrac{z}{n}\right) ]^2  p(w-z) \lesssim n^{-1}. 
\end{align*}
where above we applied Proposition \ref{tight2condaux}  for $G_{+}$ and $G_{-}$. Moreover, we note the sum over $\{w,z \} \in \mcb S$ in  last line is bounded by a constant depending on $G$. This ends the proof. 
\end{proof}
\section{Characterization of limit points} \label{seccharac}

 Since there is at most one particle per site, according to \cite{kipnis1998scaling}, $\mathbb{Q}$ is concentrated on trajectories $\pi_t(du)$ which are absolutely continuous with respect to the Lebesgue measure, that is
 \begin{align*}
\mathbb{Q} \Big(\pi_{ \cdot} \in \mcb D([0,T], \mcb{M}^+(\mathbb{R})):\pi_t(du)= \varrho(t,u)du, \forall t \in [0,T]  \Big) = 1.
\end{align*}
In this section, our goal is to prove that $\mathbb{Q}$ is concentrated in trajectories $\varrho$ satisfying the first condition of weak solutions of \eqref{eqhyddifreal} when $\sigma_{\mcb S}^2 < \sigma^2$, or $\sigma_{\mcb S}^2 = \sigma^2$ and $0 \leq \beta <1$; and satisfying the first condition of weak solutions of \eqref{eqhyddifrob} with $\kappa=\frac{2m\alpha}{\sigma^2}$ when {$\sigma_{\mcb S}^2 = \sigma^2$ and} $\beta=1$; and \eqref{eqhyddifrob} with $\kappa=0$ when {$\sigma_{\mcb S}^2 = \sigma^2$ and} $\beta > 1$.

\subsection{Characterization of limit points without a slow barrier}
In this subsection, our goal is to prove that $\mathbb{Q}$ is concentrated in trajectories that satisfy the first condition of weak solutions of \eqref{eqhyddifreal}, when we assume  $\sigma_{\mcb S}^2=\sigma^2$ and $0 \leq \beta  <1$ or $\sigma_{\mcb S}^2 < \sigma^2$.

\begin{prop} \label{caraclimsembarlen}
Under the conditions of Theorem \ref{hydlim}, we have
\begin{align*}
\mathbb{Q} \Big(\pi_{ \cdot} \in \mcb D([0,T], \mcb{M}^+(\mathbb{R})): F_{\textrm{Dif}}(t,\varrho,G,\mcb g) = 0, \forall t \in [0,T], \forall G \in  \mcb S_{\textrm{Dif}} \Big) = 1.
\end{align*}
\end{prop}
\begin{proof}
The proof ends as long as we show, for any $\delta >0$ and $G \in \mcb S_{{Dif}}$, that
\begin{align}\label{eq:cond_1}
\mathbb{Q} \Big(\pi_{ \cdot} \in \mcb D([0,T], \mcb{M}^+(\mathbb{R})): \sup_{0 \leq t \leq T} |F_{Dif}(t,\varrho,G,\mcb g)| > \delta  \Big) = 0.
\end{align}
Hereafter, we denote $\bar{\varrho}:=\varrho-a$. First, we write $F_{Dif}(t,\varrho,G,\mcb g)$ as
\begin{align}
  F_{Dif}(t,\varrho,G, \mcb g) 
=&\int_{\mathbb{R}} \varrho(t,u) G(t,u) du - \int_{\mathbb{R}} \mcb g(u) G(0,u) du -   \int_0^t \int_{\mathbb{R}} \varrho(s,u) \Big[ \dfrac{\sigma^2}{2} \Delta + \partial_s \Big] G(s,u) du ds \nonumber \\
-& \mathbbm{1}_{\beta \geq 1} \frac{\sigma_{\mcb S}^2}{2}  \int_0^{t} \partial_u G(s,0)[ \varrho(s,0^{+}) - \varrho(s,0^{-})] ds. \label{fdifind}
\end{align}
Before we go on, we explain why we added the last term in  last display. For  $0 \leq \beta  <1$ that term is equal to zero, but when $\beta\geq 1$,  from Proposition \ref{estenergsembarlenfor}, we know that the measure $\mathbb{Q}$ is concentrated on trajectories $\varrho$ such that $\bar{\varrho} \in L^2 \big(0, T ;  \mcb{H}^1 (\mathbb{R}) \big)$, which means that for almost every $s \in [0,T]$, $\bar{\varrho}(s, \cdot) \in \mcb{H}^1 (\mathbb{R})$ and $\bar{\varrho}(s,0^{+})=\bar{\varrho}(s,0^{-})$ and therefore  that term is again equal to zero. Nevertheless, writing it in that form, allows, for $\beta \geq 1$, to compare it with its discrete analogue, namely the term   $ \int_0^{t }\partial_u G(s,0) [ \eta_s^{\rightarrow n \varepsilon}(0) - \eta_s^{\leftarrow n \varepsilon}(0)]ds$, which in principle is not equal to zero. \textcolor{black}{Above we define for $\ell \geq 1 \in \mathbb{N}$ the empirical averages on a box of size $\ell$ around $0$, given by:
\begin{equation} \label{medemp}
\eta^{\rightarrow \ell}(0):=\frac{1}{\ell} \sum_{y=1}^{\ell} \eta(y)   \; \; \text{and} \; \; \eta^{\leftarrow \ell}(0):=\frac{1}{\ell} \sum_{y=-\ell}^{-1} \eta(y),
\end{equation}
where we interpreted $\varepsilon n$ as $\lfloor \varepsilon n\rfloor$. 
}

Then, by Proposition \ref{convbound} we are able to link those two terms in the regime $\beta \geq 1$ and $\sigma_{\mcb S}^2 < \sigma^2$.

To simplify notation in what follows, we erase $\pi_{\cdot}$ from the sets where we  look at. We observe that, in the case $\beta\geq 1$, due to the boundary terms  $\varrho(s,0^-)$ and $\varrho(s,0^+)$ that appear in last display, we deal with sets which are not open in the Skorohod topology and therefore we are not be able to use directly Portmanteau's Theorem. In order to avoid this problem, for every $\varepsilon>0$ we define two approximations of the identity given by
\textcolor{black}{
$
\iota_{\varepsilon}^{0^+}(u):= \varepsilon^{-1} \mathbbm{1}_{(0, \varepsilon]}(u); \; \; \iota_{\varepsilon}^{0^-}(u):=\varepsilon^{-1} \mathbbm{1}_{[-\varepsilon,0)}(u). 
$
}
Summing and subtracting to $\varrho(s,0+)$ (resp., $\varrho(s,0-)$) the mean $\langle \pi_s, \iota_{\varepsilon}^{0^+} \rangle$ (resp., $\langle \pi_s, \iota_{\varepsilon}^{0^-} \rangle$) we can bound the probability in \eqref{eq:cond_1} by the sum of the following four terms:
\begin{align}
\mathbb Q\Big( \sup_{0 \leq t \leq T} \Big|\int_{\mathbb{R}} \varrho(t,u) G(t,u) du -&\int_{\mathbb{R}} \varrho(0,u) G(0,u) du  
-  \int_0^t \int_{\mathbb{R}} \varrho(s,u) \Big[ \dfrac{\sigma^2}{2} \Delta + \partial_s \Big] G(s,u) du ds \nonumber  \\
- &  \mathbbm{1}_{\beta \geq 1} \frac{\sigma_{\mcb S}^2}{2}  \int_0^t \partial_u G(s,0) [   \langle \pi_s, \iota_{\varepsilon}^{0^+} \rangle - \langle \pi_s, \iota_{\varepsilon}^{0^-} \rangle] ds \Big| > \dfrac{\delta}{4} \Big), \label{f1term1dif}
\end{align}
\begin{equation} \label{f1term2dif}
\mathbb{Q} \Big(  \Big| \int_{\mathbb{R}} \left[ \varrho(0,u) - \mcb g(u) \right] G(0,u) du \Big| > \dfrac{\delta}{4}  \Big),
\end{equation}
\begin{equation} \label{f1term3dif}
\mathbb{Q} \Big(  \mathbbm{1}_{\beta \geq 1} \frac{\sigma_{\mcb S}^2}{2} \sup_{0 \leq t \leq T} \Big| \int_{0}^{t} \partial_u G(s,0) [ \varrho(s,0^{+}) - \langle \pi_s, \iota_{\varepsilon}^{0^+} \rangle ]  ds \Big| > \dfrac{\delta}{4}  \Big),
\end{equation}
and
\begin{equation} \label{f1term4dif}
\mathbb{Q} \Big( \mathbbm{1}_{\beta \geq 1} \frac{\sigma_{\mcb S}^2}{2} \sup_{0 \leq t \leq T} \Big| \int_{0}^{t} \partial_u G(s,0) [ \varrho(s,0^{-}) - \langle \pi_s, \iota_{\varepsilon}^{0^-} \rangle]  ds \Big| > \dfrac{\delta}{4}  \Big).
\end{equation}
Above, we used \eqref{fdifind}. Since $\mathbb{Q}$ is a limit point of $(\mathbb{Q}_n)_{n \geq 1 }$, which is induced by  $(\mu_n)_{n  \geq 1}$ and it is associated to the profile $\mcb g$,the expression in \eqref{f1term2dif} is equal to zero. We observe that the expressions in \eqref{f1term3dif} and \eqref{f1term4dif} go to zero when $\varepsilon \rightarrow 0^+$.
Therefore, it remains only to consider \eqref{f1term1dif}. We still cannot use Portmanteau's Theorem, since the functions $\iota_{\varepsilon}^{0^+}$ and $\iota_{\varepsilon}^{0^-}$ are not continuous. In order to do so, we can approximate each one of these functions by continuous functions in such a way that the error vanishes as $\varepsilon \rightarrow 0^+$. In this way, we will be able to bound \eqref{f1term1dif} from above by
\begin{align}
& \liminf_{n \rightarrow \infty} \mathbb{Q}_n \Big(  \sup_{0 \leq t \leq T} \Big| \int_{\mathbb{R}} \varrho(t,u) G(t,u) du -\int_{\mathbb{R}} \varrho(0,u) G(0,u) du   
-  \int_0^t  \int_{\mathbb{R}} \varrho(s,u)  \dfrac{\sigma^2}{2} \Delta G(s,u) du ds \nonumber \\&- \int_0^t  \int_{\mathbb{R}} \varrho(s,u)  \partial_s G(s,u) du ds
-  \mathbbm{1}_{\beta \geq 1} \frac{\sigma_{\mcb S}^2}{2} \int_0^t \partial_u G(s,0) [  \langle \pi_s, \iota_{\varepsilon}^{0^+} \rangle  - \langle \pi_s, \iota_{\varepsilon}^{0^-} \rangle] ds \Big| > \dfrac{\delta}{16} \Big). \label{f1term1adif}
\end{align}
From the definition of $\mcb M_{t}^{n}(G)$ in \eqref{defMnt}, we can bound \eqref{f1term1adif} from above by the sum of the next two terms
\begin{equation} \label{f1term1bdif}
\liminf_{n \rightarrow \infty} \mathbb{P}_{\mu_n}  \Big(  \sup_{0 \leq t \leq T} |\mcb M_{t}^{n}(G)| > \dfrac{\delta}{32} \Big),
\end{equation}
and
\begin{align}
 \liminf_{n \rightarrow \infty}  \mathbb{P}_{\mu_n} \Big(  \sup_{0 \leq t \leq T}  \Big|  \int_{0}^{t}  n^2  \mcb L_{n}\langle \pi_{s}^{n},G_s \rangle ds &- \int_0^t  \dfrac{\sigma^2}{2}  \langle \pi_s^n, \Delta G_s \rangle ds \nonumber \\
- &  \mathbbm{1}_{\beta \geq 1} \frac{\sigma_{\mcb S}^2}{2} \int_0^t \partial_u G(s,0) [ \eta_s^{\rightarrow n \varepsilon}(0)  - \eta_s^{\leftarrow n \varepsilon}(0)]  ds \Big| > \dfrac{\delta}{32} \Big). \label{f1term1cdif}
\end{align}
From Doob's inequality and Proposition \ref{tightcond2dif} the expression in \eqref{f1term1bdif} is equal to zero. 
From Proposition \ref{gendif}, we can bound \eqref{f1term1cdif} from above by the sum of the next three terms
\begin{equation} \label{f1term1c1dif}
\limsup_{n \rightarrow \infty}  \mathbb{P}_{\mu_n}  \Big( \sup_{0 \leq t \leq T} \Big| \int_0^t \Big\{ \dfrac{1}{n} \sum _{z } n^2 \mcb {K}_n G_s \left(\tfrac{z}{n} \right)  \eta_{s}^{n}(z) - \frac{\sigma^2}{2} \langle \pi_s^n, \Delta G_s \rangle \Big\} ds \Big|> \dfrac{\delta}{96}  \Big),
\end{equation}
\begin{align}
 \limsup_{\varepsilon \rightarrow 0^+} \limsup_{n \rightarrow \infty}  \mathbb{P}_{\mu_n}  \Big( \sup_{0 \leq t \leq T} \Big|& \int_{0}^{t} \Big\{ \frac{n}{2}    \sum_{\{y,z\} \in \mcb S}  [G(s,\tfrac{y}{n}) - G(s,\tfrac{z}{n})] p(z-y)  [\eta_s^n(y)-\eta_s^n(z)] \nonumber \\
- & \mathbbm{1}_{\beta \geq 1} \frac{\sigma_{\mcb S}^2}{2} \partial_u G(s,0)[ \eta_s^{\rightarrow n \varepsilon}(0) - \eta_s^{\leftarrow n \varepsilon}(0)] \Big\} ds\Big|> \dfrac{\delta}{96}  \Big)\label{f1term1c2dif}
\end{align}
and
\begin{equation} \label{f1term1c3dif}
 \limsup_{\varepsilon \rightarrow 0^+} \limsup_{n \rightarrow \infty}  \mathbb{P}_{\mu_n}  \Big( \sup_{0 \leq t \leq T} \Big| \int_{0}^{t} \frac{\alpha n^{1-\beta}}{2}  \sum_{\{y,z\} \in \mcb S}  [G(s,\tfrac{y}{n}) - G(s,\tfrac{z}{n})] p(z-y)  [\eta_s^n(z)-\eta_s^n(y)] ds \Big|> \dfrac{\delta}{96}  \Big).
\end{equation}
From Proposition \ref{convdisc} and Markov's inequality,  \eqref{f1term1c1dif} is equal to zero. From Proposition \ref{convbound} and Markov's inequality,  \eqref{f1term1c2dif}  and \eqref{f1term1c3dif} are both equal to zero.
\end{proof}

\subsection{Characterization of limit points with a slow barrier} \label{subseccarac}

In this section, our goal is to prove that $\mathbb{Q}$ is concentrated in trajectories that satisfy the first condition of weak solutions of \eqref{eqhyddifrob}, when {$\sigma_{\mcb S}^2 = \sigma^2$ and} $\beta \geq 1$.
\begin{prop} \label{caraclimcombarlen}
Under the conditions of Theorem \ref{hydlim}, we have
\begin{align*}
\mathbb{Q} \Big(\pi_{ \cdot} \in \mcb D([0,T], \mcb{M}^+(\mathbb{R})): F_{\textrm{Rob}}(t, \varrho ,G, \mcb  g, \frac{2m\alpha}{\sigma^2} \mathbbm{1}_{\beta = 1}) = 0, \forall t \in [0,T], \forall G \in \mcb S_{\textrm{Rob}}  \Big) = 1.
\end{align*}
\end{prop}
\begin{proof}
As in the previous proof,  it is enough to verify \eqref{eq:cond_1} for any $G \in \mcb S_{\textit{Rob}} $  and for $F_{\textit{Rob}}(t,\varrho,G,\mcb g,\frac{2m\alpha}{\sigma^2} \mathbbm{1}_{\beta = 1}) $. Using exactly the same arguments as above, we are left to prove that 
\begin{align}
 \liminf_{n \rightarrow \infty} \mathbb{Q}_n &\Big(  \sup_{0 \leq t \leq T} \Big| \int_{\mathbb{R}} \varrho(t,u) G(t,u) du -\int_{\mathbb{R}} \mcb g(u) G(0,u) du  
- \int_0^t  \int_{\mathbb{R}} \varrho(s,u)  \dfrac{\sigma^2}{2} \Delta G(s,u) du ds\nonumber\\
 -& \int_0^t  \int_{\mathbb{R}} \varrho(s,u)  \partial_s G(s,u) du ds
+   \frac{\sigma^2}{2} \int_0^t [ \partial_u G(s,0^{-})   \langle \pi_s, \iota_{\varepsilon}^{0^-} \rangle  - \partial_u G(s,0^{+})  \langle \pi_s, \iota_{\varepsilon}^{0^+} \rangle ] ds \nonumber \\
+ &  m \alpha \mathbbm{1}_{\beta = 1} \int_0^t  [    \langle \pi_s, \iota_{\varepsilon}^{0^+} \rangle  -  \langle \pi_s, \iota_{\varepsilon}^{0^-} \rangle  ]  [  G(s,0^{+})    -  G(s,0^{-})  ] ds \Big| > \dfrac{\delta}{16} \Big)=0.\label{f2term1adif}
\end{align}
 From the definitions of $\mcb M_{t}^{n}(G)$ in \eqref{defMnt}, $ \eta_s^{\rightarrow n \varepsilon}(0)$ and $ \eta_s^{\leftarrow n \varepsilon}(0)$ in  \eqref{medemp}, we can bound the probability in \eqref{f2term1adif} from above by the sum of the next two terms
\begin{equation} \label{f2term1bdif}
\liminf_{n \rightarrow \infty}  \mathbb{P}_{\mu_n}  \Big(  \sup_{0 \leq t \leq T} |\mcb M_{t}^{n}(G)| > \dfrac{\delta}{32} \Big)
\end{equation}
and
\begin{align}
 \liminf_{n \rightarrow \infty}  \mathbb{P}_{\mu_n} &\Big(  \sup_{0 \leq t \leq T}  \Big|  \int_{0}^{t}  n^2  \mcb L_{n}\langle \pi_{s}^{n},G_s \rangle ds - \int_0^t  \dfrac{\sigma^2}{2}  \langle \pi_s^n, \Delta G_s \rangle ds \nonumber \\
+ &  \frac{\sigma^2}{2} \int_0^t [ \partial_u G(s,0^{-})  \eta_s^{\leftarrow n \varepsilon}(0)   - \partial_u G(s,0^{+}) \eta_s^{\rightarrow n \varepsilon}(0)  ] ds \nonumber  \\
- &   m \alpha \mathbbm{1}_{\beta = 1} \int_0^t [  G(s,0^{-})    -  G(s,0^{+})  ] [   \eta_s^{\rightarrow n \varepsilon}(0)   -\eta_s^{\leftarrow n \varepsilon}(0)   ]   ds \Big| > \dfrac{\delta}{32} \Big). \label{f2term1cdif}
\end{align}
We get from Doob's inequality and Proposition \ref{tightcond2dif}  that  \eqref{f2term1bdif} is equal to zero.
From Proposition \ref{gendif}, we can bound \eqref{f2term1cdif} from above by the sum of the next two terms 
\begin{align} 
\limsup_{\varepsilon \rightarrow 0^+} \limsup_{n \rightarrow \infty}  &\mathbb{P}_{\mu_n}  \Big( \sup_{0 \leq t \leq T} \Big|\int_{0}^{t} \Big\{  \alpha n^{1-\beta}  \sum_{\{x,z\} \in  \mcb S}  [G(s,\tfrac{x}{n}) - G(s,\tfrac{z}{n}) ]  p(x-z)  \eta_s^n(z) \nonumber \\
-& m \alpha \mathbbm{1}_{\beta = 1} [G(s,0^{-})- G(s,0^{+})]  [ \eta_s^{\rightarrow n \varepsilon}(0) - \eta_s^{\leftarrow n \varepsilon}(0)] \Big\}   ds \Big|> \dfrac{\delta}{64}  \Big) \label{f2term1c3dif}
\end{align}
and
\begin{align}
 \limsup_{\varepsilon \rightarrow 0^+} \limsup_{n \rightarrow \infty}  &\mathbb{P}_{\mu_n}  \Big( \sup_{0 \leq t \leq T} \Big| \int_{0}^{t} \Big\{  \sum_{\{x,z\} \in \mcb  F} n [G(s,\tfrac{x}{n}) - G(s,\tfrac{z}{n}) ]  p(x-z)  \eta_s^n(z) \nonumber \\
-&   \frac{\sigma^2}{2} \Big(  \partial_u G(s,0^{+})   \eta_s^{\rightarrow \varepsilon n}(0)  -\partial_u G(s,0^{-})   \eta_s^{\leftarrow \varepsilon n}(0)  +  \frac{1}{n} \sum_{z } \Delta G (s, \tfrac{z}{n})\eta_s^n(z) \Big) \Big\}   ds \Big|> \frac{\delta}{64}  \Big)\label{f2term1c2dif}
\end{align}
From Proposition \ref{convrob} and Markov's inequality,  \eqref{f2term1c3dif} is equal to zero.
From Proposition \ref{convneu} and Markov's inequality,  \eqref{f2term1c2dif} is equal to zero. 
\end{proof}
\section{Energy estimates}  \label{secenerest}
In this section, our goal is to prove that $\varrho$ satisfies the second condition of weak solutions of \eqref{eqhyddifreal} and \eqref{eqhyddifrob}, depending on the values of $\beta$ and $\sigma^2 - \sigma_{\mcb S}^2$. We begin with an important result that does not depend on the dynamics and therefore it holds for any value of $\beta\geq 0$ and $\sigma^2 - \sigma_{\mcb S}^2$. The next result only depends on the bound for $H( \mu_n | \nu_a)$, where $\mu_n$ is a probability measure on $\Omega$ and $\nu_a$ is the Bernoulli product measure of parameter $a$ \textcolor{black}{introduced in Theorem \ref{hydlim}}. Recall that $\bar{\varrho}:=\varrho-a$.
\subsection{Static energy estimates}

We define the linear functional $l_{\varrho,1}$ on 
$C_c^{0,0}  ( [0, T] \times \mathbb{R})$ by
\begin{align*}
l_{ \varrho,1 }(G):= \int_0^T \int_{\mathbb{R}} \bar\varrho(t,u)  G(t,u)  du dt. 
\end{align*}
\begin{prop} \label{estenergstat0}
$l_{\varrho,1}$ is $\mathbb{Q}-$ almost surely continuous.
\end{prop}
\begin{proof}
The proof is strongly inspired by  Section 4 of \cite{jara2009hydrodynamic}.  We define $M_{a}: \mathbb{R} \rightarrow \mathbb{R}$ by
$
M_{a}(\theta) := \mathbb{E}_{\nu_{a}}[ e^{\theta \eta(0)} ]  = a  e^{\theta } + (1-a)
$
and  $h: \mathbb{R} \rightarrow \mathbb{R}$ by
\begin{align*}
h(\theta): = \log \big( M_{a}(\theta) \big) = \log \big(  a  e^{\theta } + (1-a) \big), \forall \theta \in \mathbb{R}.
\end{align*}
For any sequence of random variables $X_1, \ldots, X_k$, we have that
\begin{align} \label{estimate}
\log \big( \mathbb{E} [ \exp \{ \max_{j=1, \ldots, k}  X_j \} ] \big)  \leq \log \big( k \max_{j=1, \ldots, k} \mathbb{E} [ \exp \{   X_j \}] \big) .
\end{align}
Let $ (G_k)_{k \geq 1}$ be a sequence of functions in 
$C_c^{0,0} \big( [0, T] \times \mathbb{R} \big)$.
For a function $G: \mathbb{R} \rightarrow \mathbb{R}$, let us define $J_n(G): = n^{-1} \sum_{x \in \mathbb{Z}} h \big( G ( \tfrac{x}{n} ) \big)$ for every $n \geq 1$ and $J(G): = \int_{\mathbb{R}} h \big( G ( u ) \big) du$.  We observe that
\begin{align*}
&\mathbb{E}_{\mu_n} \Big[ \max_{j=1, \ldots, k} \Big\{ \frac{1}{T} \int_0^T \big[ \pi_t^n \big(G_j (t, \cdot) \big) - J_n \big ( G_j (t, \cdot ) \big) \big] dt \Big\} \Big] \\
\leq & C_a + \frac{\log(k)}{n} +   \frac{1}{n} \log \Big(   \max_{j=1, \ldots, k} \mathbb{E}_{\nu_{a}} \Big[ \exp \Big\{ n   \frac{1}{T} \int_0^T \big[ \pi_t^n \big(G_j (t, \cdot) \big) - J_n \big ( G_j (t, \cdot ) \big) \big] dt \Big\} \Big]  \Big)\\
\leq & C_a + \frac{\log(k)}{n} +   \frac{1}{n} \log \Big(   \max_{j=1, \ldots, k} \frac{1}{T} \int_0^T \mathbb{E}_{\nu_{a}} \Big[ \exp \Big\{   n \big[ \pi_t^n \big(G_j (t, \cdot) \big) - J_n \big ( G_j (t, \cdot ) \big) \big]  \Big\} \Big] dt  \Big)  \\
= &  C_a + \frac{\log(k)}{n}.
\end{align*}
In the first inequality we used  the entropy inequality and \eqref{estimate}. In the second inequality we used Jensen's and Fubini's theorem. The last equality follows from the fact $\nu_a$ is a product measure. 
We know that $\mathbb{Q}$ is the limit point of some subsequence $\mathbb{Q}_{n'}$. Taking the limit when $n' \rightarrow \infty$, we can conclude that
\begin{align*}
\mathbb{E}_{\mathbb{Q}} \Big[ \sup_G  \int_0^T  \int_{\mathbb{R}}  \big[\varrho(t,u) G(t,u)-  h \big( G(t,u) \big) \big] du  dt \Big]\leq C_aT,
\end{align*}
 where the supremum is over $G \in C_c^{0,0} \big( [0,T] \times \mathbb{R}  \big)$. Fix $G \in C_c^{0,0} \big( [0,T] \times \mathbb{R}  \big)$.  Now observe that if  $f: \mathbb{R} \rightarrow \mathbb{R}$ is defined on $  x \in \mathbb{R}$ by $f(x) = \log (a e^{x} + 1 - a ) - ax=h(x)-ax$, then a simple computation shows that there exists $C > 0$ such that $f(x) \leq C x^2, \forall x \in \mathbb{R}.$
Therefore, there exists $C>0$ such that
\begin{align*}
l_{ \varrho,1 }(G)=& \int_0^T \int_{\mathbb{R}}  \big[\varrho(t,u) G(t,u)-  h \big( G(t,u) \big) \big] du dt +  \int_0^T \int_{\mathbb{R}}  \big[  h \big( G(t,u) \big) - a   G(t,u) \big] du dt \\
\leq & \int_0^T \int_{\mathbb{R}}  \big[\varrho(t,u) G(t,u)-  h \big( G(t,u) \big) \big] du dt +  \int_0^T \int_{\mathbb{R}}  C [   G(t,u) ]^2 du dt, 
\end{align*}
 which leads to
 \begin{align*}
 l_{ \varrho,1 }(G) - C  \| G \|^2_{2,[0,T] \times \mathbb{R} } \leq \int_0^T \int_{\mathbb{R}}  \big[\varrho(t,u) G(t,u)-  h \big( G(t,u) \big) \big] du dt.
 \end{align*}
 Since $G \in C_c^{0,0} \big( [0,T] \times \mathbb{R}  \big)$ is arbitrary, we conclude that
 \begin{align*}
\mathbb{E}_{\mathbb{Q}} \left[ \sup_{G} \Big\{ l_{\varrho,1}(G) - C   \|G\|^2_{2,[0,T] \times \mathbb{R} } \Big\} \right] \leq  C_a T,
\end{align*}
which leads to the desired result.
\end{proof}
Since $C_{c}^{0,0} \big( [0, T] \times \mathbb{R}\big )$ is dense in $L^2 \big( [0, T] \times \mathbb{R}\big )$, from last proposition and Riesz's Representation Theorem we conclude that 
\begin{align}\label{estenergstat}
\mathbb{Q} \Big( \pi_{ \cdot} \in \mcb D \big([0,T], \mcb{M}^+(\mathbb{R}) \big): \int_0^T \int_{\mathbb{R}} [\bar{\varrho}(t,u) ]^2   du dt < \infty \Big) = 1.
\end{align}

\subsection{Energy estimates without a slow barrier}

In this subsection our goal is to prove the next result.

\begin{prop} \label{estenergsembarlenfor}
Assume that $\sigma_{\mcb S}^2 < \sigma^2$. Then
\begin{align*}
\mathbb{Q} \Big (\pi_{ \cdot} \in \mcb D([0,T], \mcb{M}^+(\mathbb{R})): \bar{\varrho} \in L^2 \big(0, T ; \mcb{H}^1  ( \mathbb{R} ) \big)  \Big) = 1.
\end{align*}
\end{prop}

Before we prove last result we need the following auxiliary result. To that end, we define the linear functional $l_{\varrho,2}$ on {$C_c^{0,\infty} ( [0, T] \times \mathbb{R}) \subset L^2([0,T] \times \mathbb{R})$} by
\begin{align*}
l_{ \varrho,2 }(G):=&   \int_0^T \int_{\mathbb{R}} \partial_u G(t,u) \bar{\varrho}(t,u) du dt.
\end{align*}

\begin{prop} \label{estenergdifreal}
 $l_{\varrho,2}$ is $\mathbb{Q}-$ almost surely continuous.
\end{prop}

{Before we start the proof of last result we make the following observation. At this point, the reader might be puzzled about  the similarity between the statements of  Propositions \ref{estenergstat0} and \ref{estenergdifreal}. In fact, we note that by replacing, in the  proof of the previous proposition, the function $G$ by its derivative  $\partial_u G$, then and at the end of the argument we would obtain a comparison with the $L^2$-norm of  $\partial_u G$ and not of $G$ which would not allow us to conclude the almost sure continuity of the functional $l_{\varrho,2}$.   For this reason we need to redo the proof and in this case we rely on the dynamics of the model.}
\begin{proof}
Since $\sigma_{\mcb S}^2 < \sigma^2$, there exist $x_1<0, x_2 \geq 0$ with $\{x_1, x_2 \} \in \mcb F$ and $p(x_2 - x_1) >0$; the bond $\{x_1, x_2 \}$ is a ``bridge" between \textcolor{black}{ $\mathbb{Z}_{-}^{*}$ and $\mathbb{N}$}. In particular, there exists at least one path which allows moving a particle from $-1$ to $1$ only through fast bonds: first we go from $-1$ to $x_1$, then from $x_1$ to $x_2$ and finally from $x_2$ to $1$.   We remark that below we make the choice for this path using  only jumps of size $1$ (except in the bond $\{x_1, x_2 \}$) since in this way we can also treat finite-range model. Nevertheless, other paths could be considered, as long as we use a finite number of fast bonds.  

First we  fix $G \in C_{c}^{0,\infty} ( [0, T] \times \mathbb{R} )$. By Feynman-Kac's formula, we have
\textcolor{black}{
\begin{align*}
&  \tfrac{1}{n} \log \Big( \mathbb{E}_{\nu_a} \Big[ \exp \Big\{\int_0^T \big[ \sum_{x } \partial_u G \left(t, \tfrac{x}{n} \right) \eta_t^n(x)  - \tilde{C} n \|G\|_{2,[0,T] \times \mathbb{R} }^2 \big] dt \Big\} \Big] \Big) \\
\leq &    \int_0^T  \sup_f \Big\{   \sum_{x  } \tfrac{1}{n} \partial_u G \left(t, \tfrac{x}{n} \right)  \langle \eta(x), f\rangle_{\nu_a} + n  \langle\mcb {L}_{n} \sqrt{f} , \sqrt{f} \rangle_{\nu_a}  - \tilde{C}  \|G \|_{2,[0,T] \times \mathbb{R} }^2 \Big\} dt,
\end{align*}}
where the supremum is carried over all the densities $f$ with respect to $\nu_{a}$ \textcolor{black}{and $\tilde{C}$ is a positive constant that will be chosen later}. Above, $\langle f, g\rangle_{\nu_a}$ is the scalar product between $f$ and $g$ in $L^2(\Omega, \nu_a)$, that is,
$\langle f, g\rangle _{\nu_a} := \int f(\eta) g(\eta) d \nu_a
$ 
and $\langle\mcb {L}_{n} \sqrt{f} , \sqrt{f} \rangle_{\nu_a}$ is the Dirichlet form.  
Since $G \in C_{c}^{0,\infty}\left( [0, T] \times \mathbb{R} \right)$, by a Taylor expansion, we get (neglecting terms of lower order with respect to $n$)
\begin{align*}
 \tfrac{1}{n} \partial_u G \left(t, \tfrac{x}{n} \right)  =  G \left(t, \tfrac{x}{n} \right) - G \left(t, \tfrac{x-1}{n} \right), \forall x >0; \; \;  \tfrac{1}{n} \partial_u G \left(t, \tfrac{x}{n} \right)  =  G \left(t, \tfrac{x+1}{n} \right) - G \left(t, \tfrac{x}{n} \right), \forall x <0. 
\end{align*}
 This leads to
\textcolor{black}{
\begin{align*}
&\sum_{x=1}^{\infty} \tfrac{1}{n} \partial_u G \left(t, \tfrac{x}{n} \right) \eta(x) =  - G(t,0) \eta(1) + \sum_{x=1}^{\infty}   G (t, \tfrac{x}{n} )[ \eta(x) - \eta(x+1) ]; \\
&\sum_{x=-\infty}^{-1} \tfrac{1}{n} \partial_u G \left(t, \tfrac{x}{n} \right) \eta(x) =   G(t,0) \eta(-1) + \sum_{x=-\infty}^{-1}   G (t, \tfrac{x}{n} )[ \eta(x-1) - \eta(x) ].
\end{align*}
We observe that in each sum there exist a boundary term with $G(t,0)$. Since we did not perform a Taylor expansion for $x=0$, we are able to collect both terms as $ G(t,0) [\eta(-1)- \eta(0)]$. Thus, we have
}
\textcolor{black}{
\begin{align*}
&  \tfrac{1}{n} \log \Big( \mathbb{E}_{\nu_a} \Big[ \exp \Big\{\int_0^T \big[ \sum_{x } \partial_u G \left(t, \tfrac{x}{n} \right) \eta_t^n(x)  - \tilde{C} n \|G\|_{2,[0,T] \times \mathbb{R} }^2 \big] dt \Big\} \Big] \Big) \\
&\leq    \int_0^T  \sup_f \Big\{   \sum_{x =1}^{\infty} G (t, \tfrac{x}{n} ) \int[ \eta(x) - \eta(x+1) ]  f(\eta) d \nu_a + \sum_{x = - \infty}^{-1} G (t, \tfrac{x}{n} )\int [ \eta(x-1) - \eta(x) ]  f(\eta) d \nu_a 
\\&+  G(t,0)\int [ \eta(-1) - \eta(1) ]  f(\eta) d \nu_a + n  \langle \mcb {L}_{n} \sqrt{f} , \sqrt{f} \rangle_{\nu_a}  - \tilde{C}  \|G\|_{2,[0,T] \times \mathbb{R} }^2 \Big\} dt \\
+& \int_0^T  \sup_f \Big\{    \tfrac{1}{n} \partial_u G \left(t, 0 \right)  \langle \eta(0), f\rangle_{\nu_a}  \Big\} dt,
\end{align*}
}
plus terms of lower order with respect to $n$. \textcolor{black}{We observe that the term in last line of last display goes to zero as $n \rightarrow \infty$}.A simple computation shows that 
 \begin{align}\label{bound}
\langle\mcb {L}_{n} \sqrt{f} , \sqrt{f} \rangle_{\nu_a} = - \dfrac{1}{2}  D_n (\sqrt{f}, \nu_{a} ),
\end{align}
where 
$D_n (\sqrt{f}, \nu_a ) := D_n^{\mcb F} (\sqrt{f}, \nu_a ) + D_n^{\mcb S} (\sqrt{f}, \nu_a )$,
with
\begin{align*}
D_n^{\mcb F}  (\sqrt{f}, \nu_a ) := \frac{1}{2} \sum_{\{ x, y \} \in \mcb F} p(y-x) I_{x,y}  (\sqrt{f}, \nu_a ) \quad \textrm{and} \quad 
 D_n^{\mcb S}  (\sqrt{f}, \nu_a ) :=  \frac{\alpha}{2n^{\beta}}  \sum_{\{ x, y \} \in \mcb S} p(y-x) I_{x,y}  (\sqrt{f}, \nu_a ),
\end{align*} and 
$
I_{x,y}  (\sqrt{f}, \nu_a ) := \int [ \sqrt{f \left( \eta^{x,y} \right) } - \sqrt{f \left( \eta \right) } ]^2 d \nu_a.
$
From this, we get for every $t \in [0,T]$:
\begin{align*}
&  \sum_{x =1}^{\infty} G (t, \tfrac{x}{n} ) \int [ \eta(x) - \eta(x+1) ] f(\eta) d \nu_a + \sum_{x = - \infty}^{-1} G (t, \tfrac{x}{n} )\int [ \eta(x-1) - \eta(x) ]  f(\eta) d \nu_a \\
+ &  G(t,0) \int [ \eta(-1) - \eta(1) ]  f(\eta) d \nu_a + n  \langle \mcb {L}_{n} \sqrt{f} , \sqrt{f} \rangle_{\nu_a} \\
\leq &  \sum_{x =1}^{\infty} |G \left(t, \tfrac{x}{n} \right)|  \Big| \int [ \eta(x) - \eta(x+1) ] f \left( \eta \right) d \nu_a \Big| + \sum_{x = - \infty}^{-1} |G \left(t, \tfrac{x}{n} \right)|  \Big| \int [ \eta(x) - \eta(x-1) ] f \left( \eta \right) d \nu_a \Big| \\
+ &  |G(t,0)| \Big| \int [ \eta(-1) - \eta(1) ]  f(\eta) d \nu_a \Big| -   \dfrac{n}{2}  D_n^{\mcb F} (\sqrt{f}, \nu_{a} ).
\end{align*} 
Through the bond $\{x_1,x_2\}$, we can go from $-1$ to $1$ only from through fast bonds following the path described in the beginning of the proof. This motivates us to write
\begin{align*}
&\Big| \int [ \eta(-1) - \eta(1) ]  f(\eta) d \nu_a \Big| \leq \Big| \int [ \eta(x_1) - \eta(x_2) ]  f(\eta) d \nu_a \Big| \\
 + & \sum_{j=x_1+1}^{-1} \Big| \int [ \eta(j) - \eta(j-1) ]  f(\eta) d \nu_a \Big| +   \sum_{j=1}^{x_2-1} \Big| \int [ \eta(j+1) - \eta(j) ]  f(\eta) d \nu_a \Big|.
\end{align*}
\textcolor{black}{Since a fixed fast bond is counted twice in $\sum_{\{y, z \} \in \mcb F}$, it holds}
\begin{align*}
&-\dfrac{n}{2}  D_n^{\mcb F}  (\sqrt{f}, \nu_{a} ) =  -\frac{n}{4}  \sum_{\{y, z \} \in \mcb F} p(z-y) I_{y,z}  (\sqrt{f}, \mu ) \\
\leq& -   \sum_{x =1}^{\infty} \frac{np(1) I_{x,x+1}   (\sqrt{f}, \nu_{a} )}{4}   -  \sum_{x = - \infty}^{-1} \frac{n p(-1) I_{x,x-1}   (\sqrt{f}, \nu_{a} )}{4}     \\
-&\frac{n}{4} p( x_2 - x_1 )  I_{x_1,x_2}  (\sqrt{f}, \nu_{a} )-  \sum_{j=x_1+1}^{-1}  \frac{np(-1) I_{j,j-1}   (\sqrt{f}, \nu_{a} )}{4} -  \sum_{j=1}^{x_2-1}  \frac{np(-1) I_{j+1,j}   (\sqrt{f}, \nu_{a} )}{4} .
\end{align*}
Now observe that from a change of variables $\eta $ to $\eta^{x,y}$ and applying Young's inequality we see that there exists a positive constant $A_{x,y}$  such that
\begin{align} \label{young}
\Big| \int [\eta(x) - \eta(y)] f(\eta) d \nu_a \Big| \leq  \frac{ I_{x,y}  (\sqrt{f}, \nu_a )}{2A_{x,y}} + 2 A_{x,y}.
\end{align}
To be precise we make the choice 
\begin{align*}
& A_{x,x+1}:= \frac{2|G \left(t, \tfrac{x}{n} \right)|}{n p(1)}, \forall x \in \mathbb{Z}_{+}^{*};\quad  A_{x,x-1}:= \frac{2|G \left(t, \tfrac{x}{n} \right)|}{n p(-1)},    \forall x \in \mathbb{Z}_{+}^{*}; \quad A_{x_1,x_2}:= \frac{2|G (t, 0)|}{n p(x_2-x_1)};\\
& A_{j,j-1}:= \frac{2|G (t, 0)|}{n p(-1)}, \forall j \in \{x_1+1, \ldots, -1\}; A_{j+1,j}:= \frac{2|G (t, 0)|}{n p(-1)}, \forall j \in \{1, \ldots, x_2-1\}.
\end{align*}
\textcolor{black}{
Therefore, by choosing $\tilde{C}:=4 \max\{ [p(1)]^{-1}, [p(x_2-x_1)]^{-1} +(x_2-x_1-2)[p(1)]^{-1}\}>0$, after an iterated application of last inequality and taking the $\limsup$ when $n \rightarrow \infty$ we get that for every $G \in C_{c}^{0,\infty} ( [0,T] \times \mathbb{R} )$
\begin{align} \label{EGnpos1}
& \limsup_{n \rightarrow \infty}   \tfrac{1}{n} \log \Big( \mathbb{E}_{\nu_a} \Big[ \exp \Big\{\int_0^T \big[ \sum_{x } \partial_u G \left(t, \tfrac{x}{n} \right) \eta_t^n(x)  - \tilde{C} n \|G\|_{2,[0,T] \times \mathbb{R} }^2 \big] dt \Big\} \Big] \Big)  \leq 0.
\end{align}
}
Now we choose a sequence $\{G_m\}_{m \geq 1}$on $C_{c}^{0,\infty} \left( [0,T] \times \mathbb{R} \right)$. Fix $m_0 \geq 1$ and define $\Phi: \mcb {D}([0,T], \mcb{M}^+) \rightarrow \mathbb{R}$ by
\textcolor{black}{
\begin{align*}
\Phi (\pi_{\cdot}) := \max_{k \leq m_0} \Big\{  \int_0^T \big[ \int_{\mathbb{R}} \partial_u G_k(t,u) d \pi(t,u)  - \tilde{C} \|G_k\|_{2,[0,T] \times \mathbb{R} }^2\big] dt \Big\} ,
\end{align*}
}
which is a continuous and bounded function for the Skorohod topology of $ \mcb {D}([0,T], \mcb{M}^+)$. Thus we have
\textcolor{black}{
\begin{align*}
\mathbb{E}_{\mathbb{Q}} \left[ \Phi \right] = \lim_{n \rightarrow \infty} \mathbb{E}_{\mu_n} \Big[  \max_{k \leq m_0} \Big\lbrace \int_0^T \big[ \frac{1}{n}  \sum_{x }  \partial_u G_k (t, \tfrac{x}{n} ) \eta_t^n(x)  - \tilde{C} \|G_k\|_{2,[0,T] \times \mathbb{R} }^2 \big] dt  \Big\rbrace \Big].
\end{align*}
}
By the entropy inequality, Jensen's inequality and  $e^{\max_{k \leq m} a_k} \leq \sum_{k=1}^m e^{a_k}$, we have
\textcolor{black}{
\begin{align*}
& \mathbb{E}_{\mu_n} \Big[  \max_{k \leq m_0} \Big\lbrace \int_0^T \big[ \frac{1}{n} \sum_{x }  \partial_u G_k (t, \tfrac{x}{n} ) \eta_t^n(x)  - \tilde{C} \|G_k\|_{2,[0,T] \times \mathbb{R} }^2 \big] dt \Big\rbrace \Big] \\
\leq & C_a + \frac{1}{n} \log   \Big( \sum_{k=1}^{m_0} \mathbb{E}_{\nu_a} \Big[   \exp \Big\{ \int_0^T \big[ \sum_{x }  \partial_u G_k (t, \tfrac{x}{n} ) \eta_t^n(x)  - \tilde{C} n \|G_k\|_{2,[0,T] \times \mathbb{R} }^2\big] dt \Big\}  \Big]  \Big).
\end{align*}
}
Since 
\begin{equation} \label{largedev}
\limsup_{n \rightarrow \infty} \dfrac{1}{n} \log (a_n + b_n ) = \max \left\lbrace \limsup_{n \rightarrow \infty} \dfrac{1}{n} \log (a_n  ), \; \limsup_{n \rightarrow \infty} \dfrac{1}{n} \log ( b_n ) \right\rbrace,
\end{equation}
from \eqref{EGnpos1}, we get 
\textcolor{black}{
\begin{align*}
& \mathbb{E}_{\mathbb{Q}} \left[ \Phi \right] = \lim_{n \rightarrow \infty} \mathbb{E}_{\mu_n} \Big[  \max_{k \leq m_0} \Big\lbrace \int_0^T \big[ \frac{1}{n} \sum_{x }  \partial_u G_k (t, \tfrac{x}{n} ) \eta_t^n(x)  - \tilde{C} \|G_k\|_{2,[0,T] \times \mathbb{R} }^2 \big] dt \Big\rbrace \Big]\leq C_a. 
\end{align*}
}
Therefore, by density and by the Monotone Convergence Theorem, 
\textcolor{black}{
\begin{align*}
\mathbb{E}_{\mathbb{Q}} \big[ \sup_{G} \{ l_{\varrho,2}(G) - \tilde{C} T   \|G \|_{2,[0,T] \times \mathbb{R} }^2 \} \big] \leq C_a < \infty,
\end{align*}
}
where the supremum above is taken on the set $C_{c}^{0, \infty} ( [0, T] \times \mathbb{R} )$. This leads to the desired result.
\end{proof}
Now we present the proof of Proposition \ref{estenergsembarlenfor}.
\begin{proof}[Proof of Proposition \ref{estenergsembarlenfor}]
We observe that $C_{c}^{0, \infty} \big( [0, T] \times \mathbb{R}\big )$ is dense in $L^2 \big( [0, T] \times \mathbb{R}\big )$. From Proposition \ref{estenergdifreal}, and Riesz's Representation Theorem we find $\xi \in L^2 \big( [0,T] \times \mathbb{R} \big)$ such that
\begin{equation} \label{riesz}
l_{ \varrho,2 }(G)= \int_0^T \int_{\mathbb{R}} \partial_u G(t,u)  \bar{\varrho}(t,u)  du dt =  \int_0^T \int_{\mathbb{R}} G(t,u) \xi(s,u) du ds, 
\end{equation}
for every $G \in  L^2 \big( [0,T] \times \mathbb{R} \big)$. Now fix $\phi \in C_c^{\infty}(\mathbb{R})$. For every $t \in [0,T]$, define $G_t: [0,T] \times \mathbb{R} \in L^2 \big( [0,T] \times \mathbb{R} \big)$ by 
$G_t(s,u) =\phi(u)\mathbbm{1}_{ (s,u) \in [0,t] \times \mathbb{R}}$. Therefore from \eqref{riesz}, for every $t \in [0,T]$ it holds
\begin{align*}
\int_0^t \Big\{ \int_{\mathbb{R}} \phi'(u)  \bar{\varrho}(t,u)  du +  \int_{\mathbb{R}} \phi(u) \big(-  \xi(t,u) \big) du \Big\} ds=0, \forall t \in [0,T].
\end{align*}
As a consequece, for almost every $s$ on $[0,T]$ we have 
\begin{align*}
\int_{\mathbb{R}} \phi'(u)  \bar{\varrho}(s,u)  du =-  \int_{\mathbb{R}} \phi(u) \big(-  \xi(s,u) \big) du.
\end{align*}
Note that  $\xi \in L^2 \big( [0,T] \times \mathbb{R} \big)$. From  \ref{estenergstat}, we obtain  $\bar{\varrho}(s, \cdot) \in L^2(\mathbb{R})$, for almost every $s\in[0,T]$. Since $\phi \in C_c^{\infty}(\mathbb{R})$ is arbitrary, it follows that $\bar{\varrho} (s, \cdot) \in \mcb{H}^1(\mathbb{R})$ and $- \xi(s, \cdot) = \partial_u \bar{\varrho}(s, \cdot)$ for almost every $s$ on $[0,T]$, which implies  $\bar{\varrho} \in L^2 \big(0, T ; \mcb{H}^1 (\mathbb{R} )  \big)$.
\end{proof}

\subsection{Energy estimates with a slow barrier}

In this section, we assume that $\sigma_{\mcb S}^2 = \sigma^2$. Then given two sites $y_1 \in \mathbb{N}, y_2 \in \mathbb{Z}_{-}^{*}$, it is not possible to move a particle from $y_1$ to $y_2$ choosing only fast bonds, which was the crucial argument used before. However, the movement inside $\mathbb{N}$ or $ \mathbb{Z}_{-}^{*}$ without using  slow bonds is always possible. Then the same argument used in the proof of Proposition \ref{estenergsembarlenfor} leads to the next result.

\begin{prop} \label{estenerbarfor}
Assume $\sigma_{\mcb S}^2 = \sigma^2$. Then
\begin{align*}
\mathbb{Q} \Big( \pi_{ \cdot} \in \mcb D \big([0,T], \mcb{M}^+(\mathbb{R}) \big): \bar{\varrho}|_{[0,T] \times \mathbb{R}_{+}^{*}} \in L^2 \big(0, T ; \mcb{H}^1( \mathbb{R}_{+}^{*})  \big) \Big) = 1,\\
\mathbb{Q} \Big( \pi_{ \cdot} \in \mcb D \big([0,T], \mcb{M}^+(\mathbb{R}) \big): \bar{\varrho}|_{[0,T] \times \mathbb{R}_{-}^{*}} \in L^2 \big(0, T ; \mcb{H}^1( \mathbb{R}_{-}^{*})  \big) \Big) = 1.
\end{align*}
\end{prop}
\textcolor{black}{A consequence of Lemma \ref{partreplemma} is the following}:
\begin{prop} \label{dircond}
Assume $\beta \in [0, 1)$. If $\mathbb{Q}$ is a limit point of the sequence $(\mathbb{Q}_n)_{n  \geq 1}$, then
\begin{align*}
\mathbb{Q} \Big( \pi_{ \cdot} \in \mcb D \big([0,T], \mcb{M}^+(\mathbb{R}) \big): \int_0^t [\varrho(s,0^{+}) - \varrho(s,0^{-})] ds = 0, \forall t \in [0,T]  \Big) = 1.
\end{align*}
\end{prop}
\begin{proof}
In order to prove the proposition, it is enough to verify, for any $\delta >0$
\begin{align*}
\mathbb{Q} \Big(\pi_{ \cdot} \in \mcb D([0,T], \mcb{M}^+(\mathbb{R})): \sup_{0 \leq t \leq T} \Big| \int_0^t [\varrho(s,0^{+}) - \varrho(s,0^{-})] ds \Big| > \delta  \Big) = 0.
\end{align*}
For almost every $s \in [0,T]$ we have $\bar{\varrho}(s,u) \in \mcb{H}^1(\mathbb{R}^{*})$ and $\varrho=\bar{\varrho}+a$ is such that $\varrho(s,\cdot)|_{\mathbb{R}_{-}^{*}}$ and $\varrho(s,\cdot)|_{\mathbb{R}_{+}}$ have continuous representatives on $(-\infty,0]$ and $[0, \infty)$, respectively. Then Lebesgue's Differentiation Theorem leads to 
\begin{align*}
\varrho(s,0^{+}) = \lim_{\varepsilon \rightarrow 0^{+}} \frac{1}{\varepsilon} \int_{0}^{\varepsilon} {\varrho}(s,u) du \quad\textrm{and }\quad  \varrho(s,0^{-}) = \lim_{\varepsilon \rightarrow 0^{-}} \frac{1}{\varepsilon} \int_{-\varepsilon}^{0} {\varrho}(s,u) du, 
\end{align*}
for almost every $s \in [0,T]$.  From the previous identities, together with Portmanteau's theorem (after replacing $\iota_\varepsilon^{0^+}$ and $\iota_\varepsilon^{0^-}$ by continuous functions as we did in the proof of Proposition \ref{caraclimsembarlen}) and then Markov's inequality, we have that
\begin{align*}
 \mathbb{Q} \Big(\pi_{ \cdot} \in \mcb D([0,T], \mcb{M}^+(\mathbb{R})):& \sup_{0 \leq t \leq T} \Big| \int_0^t [\varrho(s,0^{+}) - \varrho(s,0^{-})] ds \Big| > \delta  \Big) \\
\leq & \limsup_{\varepsilon \rightarrow 0^+} \limsup_{n \rightarrow \infty} \delta^{-1} \mathbb{E}_{\mu_n} \Big[ \Big| \int_0^T [ \eta_s^{\rightarrow  \varepsilon n}(0) -  \eta_s^{\leftarrow\varepsilon n }(0))] ds \Big|   \Big]. 
\end{align*}
Define $F:=[0,T] \rightarrow \mathbb{R}$ by $F(s)=1, \forall s \in [0,T]$ and $\theta: \mathbb{Z} \rightarrow \mathbb{R}$ by $\theta(z)= 2 p(z), \forall z \in \mathbb{Z}$. Then $F \in L^{\infty}([0,T])$ and $\theta \in L^1(\mathbb{Z})$. \textcolor{black}{From Lemma \ref{partreplemma}, the double limit} above is zero and we have the desired result.
\end{proof} 
Finally, we  state the last result of this section.
\begin{cor} \label{estenergcombarlenfor}
Assume $\sigma_{\mcb S}^2 = \sigma^2$. Then
\begin{align*}
\mathbb{Q} \Big( \pi_{ \cdot} \in D \big([0,T], \mcb{M}^+(\mathbb{R}) \big): \bar{\varrho} \in L^2 \big(0, T ; \mcb{H}^1( \mathbb{R}^{*})  \big) \Big) = 1.
\end{align*}
Moreover, if $\beta \in [0,1)$, we have
\begin{align*}
\mathbb{Q} \Big( \pi_{ \cdot} \in D \big([0,T], \mcb{M}^+(\mathbb{R}) \big): \bar{\varrho} \in L^2 \big(0, T ; \mcb{H}^1( \mathbb{R})  \big) \Big) = 1.
\end{align*}
\end{cor} 
\begin{proof}
The first statement is a direct consequence of the definition of $L^2 \big(0, T ; \mcb{H}^1( \mathbb{R}^{*})  \big)$ and Proposition \ref{estenerbarfor}. Now if $\beta \in [0,1)$ we can apply Proposition \ref{dircond} and get
\begin{align*}
\mathbb{Q} \Big( \pi_{ \cdot} \in \mcb D \big([0,T], \mcb{M}^+(\mathbb{R}) \big): \int_0^t [\varrho(s,0^{+}) - \varrho(s,0^{-})] ds = 0, \forall t \in [0,T]  \Big) = 1.
\end{align*}
From last identity, $\bar{\varrho}(s,0^{-})= \bar{\varrho}(s,0^{+})$, for almost every $s \in [0,T]$, $\mathbb{Q}$-almost surely. From Proposition \ref{H1reta}, $\bar\varrho(s, \cdot) \in \mcb{H}^1(\mathbb{R})$, for almost every $s \in [0,T]$, which leads to the second statement.  
\end{proof}

\section{Useful $L^1(\mathbb{P}_{\mu_n})$ estimates} \label{secheurwithout}

 In this section, we show some convergences in $L^1(\mathbb{P}_{\mu_n}$) that were used along the article. First we recall \eqref{medemp}. 
 
 \subsection{Replacement Lemmas} \label{secreplemma}
In this subsection, we prove Lemma \ref{partreplemma} and Lemma \ref{replemma}.  Hereinafter, we fix $C_a >0$ such that $H( \mu_n | \nu_a) \leq  C_a n, \forall n \geq 1$.

\begin{lem} \textbf{(Replacement Lemma at the origin)}  \label{partreplemma}
Let $F \in L^{\infty}([0,T])$. For every $\beta \geq 0$, it holds
\begin{equation} \label{rlpright} 
\limsup_{\varepsilon \rightarrow 0^+} \limsup_{n \rightarrow \infty} \mathbb{E}_{\mu_n} \Big[ \Big| \int_0^t F(s)   [ \eta_s^{\rightarrow\varepsilon n }(0) - \eta_{s}^{n}(0)] ds \Big|  \Big] = 0, \forall t \in [0,T].
\end{equation}
Moreover, if $\beta<1$, we also have
\begin{equation} \label{rlpleft} 
\limsup_{\varepsilon \rightarrow 0^+} \limsup_{n \rightarrow \infty} \mathbb{E}_{\mu_n} \Big[ \Big| \int_0^t F(s)  [ \eta_s^{\leftarrow\varepsilon n }(0) - \eta_{s}^{n}(0)] ds \Big|  \Big] = 0, \forall t \geq 0.
\end{equation}
\end{lem}
\begin{proof}
We present here only the proof of \eqref{rlpleft}, but we observe that the proof of \eqref{rlpright} is analogous. In order to control the expectation in \eqref{rlpleft}, we  apply  the entropy and Jensen's inequalities. Then,  from the hypothesis on the entropy of $\mu_n$ and by Feynman-Kac's formula, for every $D>0$, we can bound it from above  by 
\begin{align*}
\frac{C_a}{D} + T \sup_{f} \Big\{ \|F\|_{\infty} | \langle \eta^{\leftarrow  \varepsilon n}(0) - \eta(0) , f \rangle_{\nu_{a}} |  + \frac{n}{D} \langle \mcb L_n \sqrt{f}, \sqrt{f} \rangle_{\nu_{a}}  \Big\}, 
\end{align*}
where the supremum is carried over all the densities $f$ with respect to $\nu_{a}$. The absolute value in the first line was removed by observing that $e^{|x|} \leq e^{x}+ e^{-x}$ and also \eqref{largedev}.  Note that
\begin{align*}
&\langle  \eta^{\leftarrow  \varepsilon n}(0) - \eta(0) , f \rangle_{\nu_{a}}   = \frac{1}{\varepsilon n}  \sum_{y=-\varepsilon n}^{ -1}   \sum_{x=y+1}^{0} \int   [  \eta(x-1) -  \eta(x) ] f(\eta) d \nu_{a} .
\end{align*}

Taking $A_{-1,0} = \frac{D \|F\|_{\infty}}{n^{1-\beta} p(1)  \alpha}$ and  $A_{x-1,x}=\frac{D ||F||_{\infty}}{p(1) n}$ for every $- \varepsilon n +1 \leq x \leq -1$ in \eqref{young}, we have
\begin{align*}
\Big| \int    [  \eta(-1) -  \eta(0) ] f(\eta) d \nu_{a} \Big| \leq &  \frac{n^{1-\beta} p(1) \alpha}{2 D \|F\|_{\infty}} I_{-1,0}  (\sqrt{f}, \nu_a ) + \frac{2D \|F\|_{\infty}}{n^{1-\beta} p(1)  \alpha},
\end{align*}
\begin{align*}
\Big| \int   [  \eta(x-1) -  \eta(x) ] f(\eta) d \nu_{a} \Big| \leq \frac{p(1) n}{2 D \|F\|_{\infty}} I_{x-1,x}  (\sqrt{f}, \nu_a ) + \frac{2D \|F\|_{\infty}}{p(1) n}.
\end{align*}
Proposition \ref{bound} leads to
\begin{align*}
\frac{n}{D} \langle \mcb L_n \sqrt{f}, \sqrt{f} \rangle_{\nu_{a}} \leq& - \frac{n^{1-\beta} p(1)  \alpha}{2 D} I_{-1,0}  (\sqrt{f}, \nu_a ) - \sum_{x=- \varepsilon n+1}^{-1} \frac{ p(1) n}{2 D} I_{x-1,x}  (\sqrt{f}, \nu_a ).
\end{align*}
Then we get
\begin{align*}
\|F\|_{\infty} | \langle \eta^{\leftarrow  \varepsilon n}(0) - \eta(0) , f \rangle_{\nu_{a}} | + \frac{n}{D} \langle \mcb L_n \sqrt{f}, \sqrt{f} \rangle_{\nu_{a}}  \leq   \frac{2D (\|F\|_{\infty})^2}{n^{1-\beta} p(1)  \alpha} +  \frac{2D (\|F\|_{\infty})^2}{p(1)  } \varepsilon,
\end{align*}
for every $f$ density with respect to $\nu_{a}$ and for every $D>0$. Taking $D=\varepsilon^{-\frac{1}{2}}$, we bound the expectation in \eqref{rlpleft} by 
\begin{align*}
  \frac{C_a}{\varepsilon^{-\frac{1}{2}}}  + T \frac{2 (\|F\|_{\infty})^2 \varepsilon^{-\frac{1}{2}}}{n^{1-\beta} p(1) \alpha} + T \frac{2 (\|F\|_{\infty})^2 \varepsilon^{-\frac{1}{2}}}{p(1)  } \varepsilon,
\end{align*}
so that the proof ends by taking first the limit in $n\to+\infty$ and then $\epsilon \to 0$, and recalling that $\beta<1$. 
\end{proof}
Analogously, we can prove the more general result given below, that deals only with fast bonds and holds for every $\beta \geq 0$.
\begin{lem} \textbf{(Replacement Lemma)} \label{replemma}
Let $F \in L^{\infty}([0,T])$ and $\theta \in L^1(\mathbb{Z})$. Then for every $t \in [0,T]$, 
\begin{equation} \label{rlleft}
\limsup_{\varepsilon \rightarrow 0^+} \limsup_{n \rightarrow \infty} \mathbb{E}_{\mu_n} \Big[ \Big| \int_0^t F(s) \sum_{z =-\infty}^{-1}  \theta(z) [\eta_{s}^{n}(z) - \eta_s^{\leftarrow\varepsilon n }(0) ] ds \Big|   \Big] = 0,
\end{equation}
\begin{equation} \label{rlright}
\limsup_{\varepsilon \rightarrow 0^+} \limsup_{n \rightarrow \infty} \mathbb{E}_{\mu_n} \Big[ \Big|  \int_0^t F(s)  \sum_{z =0}^{\infty}  \theta(z) [\eta_{s}^{n}(z) - \eta_s^{\rightarrow  \varepsilon n}(0) ] ds \Big|   \Big] = 0.
\end{equation}
\end{lem}

\subsection{Convergences without a slow barrier}
In this subsection we will assume that $\sigma_{\mcb S}^2=\sigma^2$ and $0\leq \beta<1$ or $\sigma_{\mcb S}^2<\sigma^2$ and $\beta\geq 0$. In this setting, our space of test functions is $\mcb S_{\textit{Dif}}$.  Now we  analyse the behavior of \eqref{extradif} according to $\beta$, when $n$ goes to infinity. 

\begin{prop} \label{convbound}
Let $t \in [0,T]$ and $G \in \mcb S_{\textrm{Dif}}$. Then, we have
\begin{align}
 \limsup_{\varepsilon \rightarrow 0^+} \limsup_{n \rightarrow \infty} \mathbb{E}_{\mu_n} \Big[\Big| \int_{0}^{t} \Big\{ \frac{n}{2} \sum _{ \{y, z \} \in \mcb S }& \left[ G(s, \tfrac{y}{n}) - G(s, \tfrac{z}{n}) \right] p(y-z) [\eta_s(y)-\eta_s(z)] \nonumber \\
-& \mathbbm{1}_{\beta \geq 1} \frac{\sigma_s^2}{2} \partial_u G(s,0) [ \eta_s^{\rightarrow n \varepsilon}(0) - \eta_s^{\leftarrow n \varepsilon}(0)] \Big\}   ds \Big| \Big] =0. \label{eqconvbound}
\end{align}
\end{prop}
\begin{proof}
Define $F: [0,T] \rightarrow \mathbb{R}$ by $F(s)= \frac{\partial_u G(s,0)}{2}$ for every $s \in [0,T]$ and $\theta : \mathbb{Z} \rightarrow \mathbb{R}$ by
\begin{align*}
 \theta(x)=\sum _{ \{x, z \} \in \mcb S }  (z-x) p(z-x) \mathbbm{1}_{x \leq -1}+
 \sum _{ \{z, x \} \in \mcb S }  (x-z) p(x-z) \mathbbm{1}_{x \geq 0}.
\end{align*}
Therefore, $F \in L^{\infty}([0,T])$ and $\theta \in L^1(\mathbb{Z})$. Applying Proposition \ref{neum1}, for $\beta \in [0,1)$, we observe that the expectation in \eqref{eqconvbound} is bounded from above by
\begin{align*}
&   \mathbb{E}_{\mu_n} \Big[ \Big| \int_0^t F(s) \sum_{z =0}^{\infty}  \theta(z) [\eta_{s}^{n}(z) - \eta_s^{\rightarrow  \varepsilon n}(0) ] ds \Big|   \Big] + \sum_{z =0}^{\infty}  \theta(z)  \mathbb{E}_{\mu_n} \Big[ \Big| \int_0^t F(s)   [\eta_s^{\rightarrow  \varepsilon n}(0) - \eta_s^n(0)  ] ds \Big|   \Big] \\
+ & \mathbb{E}_{\mu_n} \Big[  \Big| \int_0^t F(s) \sum_{z =-\infty}^{-1}  \theta(z) [\eta_{s}^{n}(z) - \eta_s^{\leftarrow\varepsilon n }(0) ] ds \Big|   \Big] + \sum_{z=-\infty}^{-1}  \theta(z)   \mathbb{E}_{\mu_n} \Big[ \Big|  \int_0^t F(s)   [ \eta_s^{\leftarrow\varepsilon n }(0) - \eta_{s}^{n}(0)] ds  \Big| \Big] .
\end{align*}
for every $\varepsilon >0$. Then, taking  $n \rightarrow \infty$ and then $\varepsilon \rightarrow 0^{+}$,  Lemma \ref{partreplemma} and Lemma \ref{replemma} lead to the desired result.

For $\beta \in [1, \infty)$, we apply Proposition \ref{neum1}, and the expectation in \eqref{eqconvbound} becomes bounded from above by
\begin{align*}
&   \mathbb{E}_{\mu_n}  \Big[  \Big|  \int_0^t F(s)  \sum_{z=0}^{\infty} \theta(z) [|\eta_s(z) -\eta_s^{\rightarrow \varepsilon n}(0)] ds \Big|  \Big] 
+ \mathbb{E}_{\mu_n}  \Big[ \Big|  \int_0^t F(s) \sum_{z=-\infty}^{-1} \theta(z)  [ \eta_s(z)- \eta_s^{\leftarrow \varepsilon n}(0)]  ds \Big| \Big],
\end{align*}
 for every $\varepsilon >0$. Then, taking  $n \rightarrow \infty$ and then $\varepsilon \rightarrow 0^{+}$,  Lemma \ref{replemma} leads to the desired result.

\end{proof}

\subsection{Convergences with a slow barrier} \label{secheurwith}

In this subsection we will assume that $\sigma_{\mcb S}^2 = \sigma^2$ and $\beta \geq 1$. In this setting, our space of test functions is $\mcb S_{Rob}$. Our goal is to  analyse the convergence in $L^1 ( \mathbb{P}_{\mu_n})$ of  $\int_{0}^{t} n^2 \mcb L_{n}\langle \pi_{s}^{n},G(s, \cdot) \rangle ds$, by studying the behavior of \eqref{robterm} and \eqref{neuterm}. We begin with \eqref{robterm}.

\begin{prop} \label{convrob}
Assume $\sigma_{\mcb S}^2 = \sigma^2$ and $ \beta \geq 1$. Let $t \in [0,T]$ and  $G \in \mcb S_{Rob}$. Then, 
\begin{align}
 \limsup_{\varepsilon \rightarrow 0^+} \limsup_{n \rightarrow \infty} \mathbb{E}_{\mu_n} \Big[&\Big| \int_{0}^{t} \Big\{  \alpha n^{1-\beta}  \sum_{\{x,z\} \in \mcb S}  [G(s,\tfrac{x}{n}) - G(s,\tfrac{z}{n}) ]  p(x-z)  \eta_s^n(z) \nonumber \\
-& \mathbbm{1}_{\beta=1} m \alpha [G(s,0^{-})- G(s,0^{+})]  [ \eta_s^{\rightarrow n \varepsilon}(0) - \eta_s^{\leftarrow n \varepsilon}(0)] \Big\}   ds \Big| \Big] =0. \label{eqconvrob01}
\end{align}
\end{prop}
\begin{proof}
For $\beta > 1$, the expectation in \eqref{eqconvrob01} is bounded from above by a constant times $n^{1-\beta}$, and we are done. Now assume $\beta=1$. Since $\sigma_{\mcb S}^2 = \sigma^2$, we have 
\begin{align*}
&\alpha n^{1-\beta} \sum_{\{x,z\} \in \mcb S}  [G(s,\tfrac{x}{n}) - G(s,\tfrac{z}{n}) ]  p(x-z)  \eta_s^n(z) \\
=& \alpha  \sum_{z=0}^{\infty}  \sum_{x = - \infty}^{-1}  [G(s,\tfrac{x}{n}) - G(s,\tfrac{z}{n}) ]  p(x-z)  \eta_s^n(z) + \alpha  \sum_{z = - \infty}^{-1} \sum_{x=0}^{\infty}    [G(s,\tfrac{x}{n}) - G(s,\tfrac{z}{n}) ]  p(x-z)  \eta_s^n(z).
\end{align*}
Define $\theta: \mathbb{Z} \rightarrow \mathbb{R}$ by
$
\theta(y)=  \sum_{z=0}^{\infty}   p(z-y) \mathbbm{1}_{y\leq -1}+\sum_{x=-\infty}^{-1}   p(y-x) \mathbbm{1}_{ y \geq 0}.
$
We observe that $\sum_{z=-\infty}^{-1} \theta(z) = \sum_{z=0}^{\infty} \theta(z)=m$. Let $F(s)=[G(s,0^{-})- G(s,0^{+})], \forall s \in [0,T]$. Then the expectation in \eqref{eqconvrob01} can be bounded from above by
\begin{align}
 &\alpha\mathbb{E}_{\mu_n} \Big[ \Big| \int_{0}^{t} \Big\{     \sum_{z=0}^{\infty}  \sum_{x = - \infty}^{-1}  [G(s,\tfrac{x}{n}) - G(s,\tfrac{z}{n}) ]  p(x-z)  \eta_s^n(z)-  F(s)  \eta_s^{\rightarrow \varepsilon n}(0) \sum_{z=0}^{\infty} \theta(z)  \Big\}   ds \Big| \Big] \label{eqconvrob02} \\
+ & \alpha \mathbb{E}_{\mu_n} \Big[ \Big| \int_{0}^{t} \Big\{     \sum_{z = - \infty}^{-1} \sum_{x=0}^{\infty}    [G(s,\tfrac{x}{n}) - G(s,\tfrac{z}{n}) ]  p(x-z)  \eta_s^n(z)- F(s)  \eta_s^{\leftarrow  \varepsilon n}(0) \sum_{z=-\infty}^{-1} \theta(z) \Big\}   ds \Big|\Big].  \label{eqconvrob03}
\end{align}
 The expectation in \eqref{eqconvrob02} can be bounded from above by
\begin{align*}
& \mathbb{E}_{\mu_n} \Big[ \Big| \int_{0}^{t} \Big\{   \sum_{z=0}^{\infty}  \sum_{x = - \infty}^{-1} p(x-z) \eta_s^n(z) \big(  [G(s,\tfrac{x}{n}) - G(s,\tfrac{z}{n}) ]  -F(s)  \big)  \Big\} ds \Big| \Big]   \\
 &+  \mathbb{E}_{\mu_n} \Big[ \Big| \int_{0}^{t}  F(s)   \sum_{z=0}^{\infty} \theta(z)  [ \eta_s^n(z)-    \eta_s^{\rightarrow n \varepsilon}(0) ]    ds \Big| \Big] .
\end{align*}
Since $F \in L^\infty([0,T])$ and $\theta \in L^1(\mathbb{Z})$, taking  $n \rightarrow \infty$ and then $\varepsilon \rightarrow 0^{+}$, we get from Proposition \ref{lemconvrob} and Lemma \ref{replemma} that the expectation in \eqref{eqconvrob02} goes to zero. In an analogous way, the same happens with the expectation in \eqref{eqconvrob03} and we get the desired result.
\end{proof}

Now we will prove the following convergence in $L^1(\mathbb{P}_{\mu_n})$:
\begin{prop} \label{convrneu}
Let $t>0$ and $G \in \mcb S_{\textrm{Rob}}$.
Then, we have
\begin{equation} \label{convneupos}
\limsup_{\varepsilon \rightarrow 0^{+}} \limsup_{n \rightarrow \infty} \mathbb{E}_{\mu_n} \Big[ \Big| \int_{0}^{t} \Big\{ \sum_{z=0}^{\varepsilon n -1} \sum_{x=0}^{\infty} \big[ n[G(s,\tfrac{x}{n}) - G(s,\tfrac{z}{n}) ]p(x-z)  \eta_s^n(z) - \frac{\sigma^2}{2} \partial_u G(s,0^{+})   \eta_s^{\rightarrow \varepsilon n}(0)  \Big\} ds \Big| \Big] =0
\end{equation}
and
\begin{equation}  \label{convneuneg}
\limsup_{\varepsilon \rightarrow 0^{+}} \limsup_{n \rightarrow \infty} \mathbb{E}_{\mu_n} \Big[ \Big| \int_{0}^{t} \Big\{ \sum_{z=-\varepsilon n +1}^{-1} \sum_{x=-\infty}^{-1}  n[G(s,\tfrac{x}{n}) - G(s,\tfrac{z}{n}) ]p(x-z)  \eta_s^n(z) + \frac{\sigma^2}{2}  \partial_u G(s,0^{-})   \eta_s^{\leftarrow \varepsilon n}(0)   \Big\} ds \Big| \Big] =0.
\end{equation}
\end{prop}

\begin{proof}
We will prove only \eqref{convneupos}, but we observe that the proof of \eqref{convneuneg} is analogous.
Define $F:[0,T] \rightarrow \mathbb{R}$ by $F(s)=\partial_u G(s,0^{+}), \forall s \in [0,T]$ and  $\theta: \mathbb{Z} \rightarrow \mathbb{R}$ by
\begin{align*}
\theta(z)=
\begin{cases}
  \sum_{x=-\infty}^{-1} (x-z)    p(x-z) = \sum_{x=-\infty}^{2z} (x-z)    p(x-z) = \sum_{r=-\infty}^{z} r   p(r), z \leq -1, \\
 \sum_{x=0}^{\infty}  (x-z)    p(x-z) = \sum_{x=2z+1}^{\infty}  (x-z)    p(x-z) = \sum_{r=z+1}^{\infty}  r    p(r), z \geq 0.
\end{cases}
\end{align*}
Observe that in the identities above, we used the fact that $p$ is symmetric. 
Observe that $F \in L^{\infty}([0,T])$ and $\sum_{z=-\infty}^{-1} \theta(z) = \frac{\sigma^2}{2}$. Since $0 \leq \eta_s(x) \leq 1$ for all $x\in\mathbb Z$, we have
\begin{equation} \label{lemconvneu3}
\limsup_{\varepsilon \rightarrow 0^{+}} \limsup_{n \rightarrow \infty} \mathbb{E}_{\mu_n} \Big[ \Big| \int_{0}^{t} \Big\{   F(s)  \ \sum_{z=\varepsilon n}^{\infty} \theta(z)  \eta_s^n(z) \Big\} ds \Big| \Big] = 0.
\end{equation}
The expectation in \eqref{convneupos} can be bounded from above by
\begin{align*}
 & \mathbb{E}_{\mu_n} \Big[ \Big| \int_{0}^{t} \Big\{ \sum_{z=0}^{\varepsilon n -1} \sum_{x=0}^{\infty} \big[ n[G(s,\tfrac{x}{n}) - G(s,\tfrac{z}{n}) ] - F(s)(x-z) \big]  p(x-z)  \eta_s^n(z)  \Big\} ds \Big| \Big]  \\
+ & \mathbb{E}_{\mu_n} \Big[ \Big| \int_{0}^{t} \Big\{ \sum_{z=0}^{\varepsilon n -1} \sum_{x=0}^{\infty}  F(s)(x-z)   p(x-z)  \eta_s^n(z) -  F(s)   \eta_s^{\rightarrow \varepsilon n}(0) \sum_{z=0}^{\infty} \theta(z) \Big\} ds \Big| \Big] \\
\leq  &  \mathbb{E}_{\mu_n} \Big[ \Big| \int_{0}^{t} \Big\{ \sum_{z=0}^{\varepsilon n -1} \sum_{x=0}^{\infty} \big[ n[G(s, \tfrac{x}{n}) - G(s, \tfrac{z}{n}) ] - F(s)(x-z) \big]  p(x-z)  \eta_s^n(z)  \Big\} ds \Big| \Big]  \\
+ &  \mathbb{E}_{\mu_n} \Big[ \Big| \int_{0}^{t} F(s) \sum_{z=0}^{\infty} \theta(z) [  \eta_s^n(z) -  \eta_s^{\rightarrow \varepsilon n}(0)] \Big\} ds \Big| \Big] 
+   \mathbb{E}_{\mu_n} \Big[ \Big| \int_{0}^{t} \Big\{  F(s)   \sum_{z=\varepsilon n}^{\infty} \theta(z)  \eta_s^n(z) \Big\} ds \Big| \Big].
\end{align*}
Taking the $\limsup$ when $n \rightarrow \infty$ and afterwards when $\varepsilon \rightarrow 0^{+}$, Proposition \ref{lemconvneum}, Lemma \ref{replemma} and \eqref{lemconvneu3} produce the desired result.
\end{proof}
Finally we can treat \eqref{neuterm}.
\begin{prop} \label{convneu}
Assume $\sigma_{\mcb S}^2 = \sigma^2$. Let $t \in [0,T]$ and $G \in \mcb S_{\textrm{Rob}}$. 
Then, we have
\begin{align*}
 \limsup_{\varepsilon \rightarrow 0^+} \limsup_{n \rightarrow \infty} \mathbb{E}_{\mu_n} \Big[&\Big| \int_{0}^{t} \Big\{ \sum_{\{x,z\} \in \mcb F} n [G(s,\tfrac{x}{n}) - G(s,\tfrac{z}{n}) ]  p(x-z)  \eta_s^n(z) \nonumber \\
-&   \frac{\sigma^2}{2} \Big(  \partial_u G(s, 0^{+})   \eta_s^{\rightarrow \varepsilon n}(0)  -  \partial_u G(s,0^{-})   \eta_s^{\leftarrow \varepsilon n}(0)  +  \frac{1}{n} \sum_{z } \Delta G ( s, \tfrac{z}{n} )\eta_s^n(z) \Big) \Big\}   ds \Big| \Big] =0. 
\end{align*}
\end{prop}
\begin{proof}
Since $\sigma_{\mcb S}^2 = \sigma^2$, we have
\begin{align*}
& \sum_{\{x,z\} \in \mcb F} n [G(s,\tfrac{x}{n}) - G(s,\tfrac{z}{n}) ]  p(x-z)  \eta_s^n(z) \\
=& \sum_{z=0}^{\infty} \sum_{x=0}^{\infty} n [G(s,\tfrac{x}{n}) - G(s,\tfrac{z}{n}) ]  p(x-z)  \eta_s^n(z) + \sum_{z=-\infty}^{-1} \sum_{x=-\infty}^{-1}  n [G(s,\tfrac{x}{n}) - G(s,\tfrac{z}{n}) ]  p(x-z)  \eta_s^n(z). 
\end{align*}
Then the expectation in the statement of the proposition can be bounded from above by the sum of expectations in \eqref{convneupos}, \eqref{convneuneg}, \eqref{princneupos} and \eqref{princneuneg}. Taking the $\limsup$ when $n \rightarrow \infty$ and $\varepsilon \rightarrow 0^{+}$, Proposition \ref{convrneu} and Corollary \ref{princneum} lead to the desired result.
\end{proof}

	\appendix
	\section{Discrete convergences}
	\label{secuseres}
	
In this section we present several propositions which were used along the article when analysing the terms in the  time integral   of Dynkin's martingale, see Proposition \ref{gendif}.
The first result we present was useful to treat \eqref{princdif}.
\begin{prop} \label{convdisc}
For every $G \in \mcb S_{\textrm{Dif}}$ we have
\begin{align} \label{limconvdisc}
\lim_{n \rightarrow \infty} \frac{1}{n} \sum_{x } \sup_{s \in [0,T]} \Big| n^2 \mcb{K}_n G \left(s,\tfrac{x}{n} \right)  - \dfrac{\sigma^2}{2}  \Delta G  \left(s,\tfrac{x}{n} \right) \Big| =0. 
\end{align}
\end{prop}
\begin{proof}
Recall \eqref{eq:B_g} and \eqref{op_Kn}. The sum in \eqref{limconvdisc} is bounded from above by
\begin{align*}
\sum_{|x| > 2b_Gn} \sup_{s \in [0,T]} \Big| n^2 \mcb {K}_n G \left(s,\tfrac{x}{n} \right)  - \dfrac{\sigma^2}{2}  \Delta G  \left(s,\tfrac{x}{n} \right) \Big|
+\sum_{|x| \leq 2b_Gn } \sup_{s \in [0,T]} \Big| n^2 \mcb {K}_n G \left(s,\tfrac{x}{n} \right)  - \dfrac{\sigma^2}{2}  \Delta G  \left(s,\tfrac{x}{n} \right) \Big|. 
\end{align*}
To treat the leftmost term in last display, we use the fact that $ \partial_u G(s,z) = 0, \forall z: |z| \geq b_G$, plus a Taylor expansion of second order and by noting that  $\sum_{r } r^2 p(r) < \infty$, we bound it from above by
\begin{align*}
 \frac{\| \Delta G \|_{\infty}}{2n}  \sum_{|y| \leq b_G n} \sum_{|x| > 2b_G n} p(y-x)  (y-x)^2 
\leq  2b_G \| \Delta G \|_{\infty}  \sum_{|r| > b_G n} r^2  p(r),
\end{align*}
and last sum vanishes as $n\to+\infty$. By a Taylor expansion of second order we can bound the remaining  term from above by
\textcolor{black}{
\begin{align}
& \limsup_{\varepsilon \rightarrow 0^{+}} \limsup_{n \rightarrow \infty} \frac{1}{n} \sum_{x = - 2 b_G n }^{2b_G n}  \sum_{|r| \geq \varepsilon n} \frac{r^2 p(r)}{2} \sup_{s \in [0,T]}  | \Delta G \left(s, \tfrac{x}{n} + \tfrac{r}{n} \right) - \Delta G \left(s, \tfrac{x}{n} \right) | \nonumber \\
+ & \limsup_{\varepsilon \rightarrow 0^{+}} \limsup_{n \rightarrow \infty} \frac{1}{n} \sum_{x = - 2 b_G n }^{2b_G n}  \sum_{|r| < \varepsilon n} \frac{r^2 p(r)}{2} \sup_{s \in [0,T]}  | \Delta G \left(s, \tfrac{x}{n} + \tfrac{r}{n} \right) - \Delta G \left(s, \tfrac{x}{n} \right) | \nonumber. 
\end{align}
The term in first line goes to zero since $\limsup_{n \rightarrow \infty} \sum_{|r| \geq \varepsilon n} r^2 p(r) = 0$, for every $\varepsilon >0$. To finish it is enough to bound the term in second line from above by
\begin{align*}
  &\sum_{r } \frac{r^2 p(r)}{2}  \limsup_{n \rightarrow \infty} \frac{1}{n} \sum_{x = - 2 b_G n }^{2b_G n} \limsup_{\varepsilon \rightarrow 0^{+}} \sup_{(s,u) \in [0,T] \times \mathbb{R}, |v| \leq \varepsilon}|\Delta G(s,u+v) - \Delta G(s,u)| \\
  &\lesssim \lim_{\varepsilon \rightarrow 0^{+}} \sup_{(s,u) \in [0,T] \times \mathbb{R}, |v| \leq \varepsilon}|\Delta G(s,u+v) - \Delta G(s,u)|=0.
\end{align*}
In last line we used the uniform continuity of $\Delta G$. This ends the proof.
}
\end{proof}
Now we present a result that was useful to treat \eqref{extradif}.
\begin{prop} \label{neum1}
For $t \in [0,T]$ and $G \in \mcb S_{\textrm{Dif}}$, it holds
\begin{align*}
 \limsup_{n \rightarrow \infty} \sup_{s \in [0,T]} \Big|   \sum _{\{y,z\} \in \mcb S} p(y-z) \big[ n [ G(s, \tfrac{y}{n}) - G(s, \tfrac{z}{n}) ] -  \partial_u G(s,0)(y-z) \big] \Big| = 0.
\end{align*}
\end{prop}
\begin{proof}
By Taylor expansion,  for every $s \in [0,t]$, for every $\{y,z\} \in \mcb S$,
\begin{align*}
 & n[ G(s, \tfrac{y}{n}) - G(s, \tfrac{z}{n}) ] -  \partial_u G(s,0) (y-z) =    \tfrac{(y-z)z}{n} \Delta G(s, \xi_{s,z} )   + \tfrac{(y-z)^2}{2n} \Delta G(s, \xi_{s,y,z} )  ,
\end{align*}
for some appropriate choice of $\xi_{s,z}$ and $\xi_{s,y,z}$. Moreover, $\forall y,z$ such that $ \min \{ |y|, |z| \}>b_Gn$, it holds 
\begin{align*}
\tfrac{(y-z)z}{n} \Delta G (s, \xi_{s,z} )   + \tfrac{(y-z)^2}{2n} \Delta  G(s, \xi_{s,y,z} ) = (z-y)  \partial_u G(s,0).
\end{align*} 

This leads to
\begin{align*}
& \limsup_{n \rightarrow \infty} \sup_{s \in [0,T]} \Big|   \sum _{\{y,z\} \in \mcb S} p(y-z) \big[ n [ G(s, \tfrac{y}{n}) - G(s, \tfrac{z}{n}) ] -  \partial_u G(s,0) \big] \Big|  \\
\leq & 3 \| \Delta G \|_{\infty}    \lim_{n \rightarrow \infty}   \frac{1}{n} \Big[ \sum _{z =0}^{b_Gn} \sum_{y=-\infty}^{-1}   (z-y)^2   p(z-y)  + \sum _{z =b_Gn+1}^{\infty} \sum_{y=-b_Gn}^{-1}   (z-y)^2   p(z-y)  \Big]=0.
\end{align*} 
\end{proof}
Now we present a result that was useful to treat \eqref{robterm}.
\begin{prop} \label{lemconvrob}
Let $ G \in \mcb S_{\textrm{Rob}}$. Then, 
\begin{align} 
\lim_{n \rightarrow \infty} \sup_{s \in [0,T]} \Big| \sum_{z=0}^{\infty}  \sum_{x = - \infty}^{-1} p(x-z) \big(  [G(s,\tfrac{x}{n}) - G(s,\tfrac{z}{n}) ]  -[G(s,0^{-}) - G(s,0^{+}) ] \big) \Big|=0. \label{lemrobzpos}
\end{align}
By symmetry the same result is true if we exchange $x$ with $z$.
\end{prop}
\begin{proof}
Since $G(s,\dot) \equiv G_{-}(s,\dot)$ on $\mathbb{R}_{-}^{*}$ and $G(s,\dot) \equiv G_{+}(s,\dot)$ on $\mathbb{R}_{+}$ for every $s \in [0,T]$, \eqref{lemrobzpos} is bounded from above by
\begin{align}\label{lemconvrob1}
&  \lim_{n \rightarrow \infty} \sum_{z=0}^{\infty} \sum_{x=-\infty}^{-1} \sup_{s \in [0,t]} |G_{-}(s,\tfrac{x}{n}) - G_{-}(s,0) |  p(x-z) +  \lim_{n \rightarrow \infty} \sum_{z=0}^{\infty} \sum_{x=-\infty}^{-1} \sup_{s \in [0,t]} |G_{+}(s,0) - G_{+}(s,\tfrac{z}{n}) |  p(x-z). 
\end{align}
We can bound the limit in the leftmost term in last display by 
\begin{align}
& \limsup_{\varepsilon \rightarrow 0^{+}}  \limsup_{n \rightarrow \infty} \sum_{z=0}^{\infty} \sum_{x=- \varepsilon n + 1}^{-1} \sup_{s \in [0,t]} |G_{-}(s,\tfrac{x}{n}) - G_{-}(s,0) | p(x-z) \label{lemconvrob3} \\
+& \limsup_{\varepsilon \rightarrow 0^{+}} \limsup_{n \rightarrow \infty} \sum_{z=0}^{\infty} \sum_{x=-\infty}^{- \varepsilon n}  \sup_{s \in [0,t]} |G_{-}(s,\tfrac{x}{n}) - G_{-}(s,0) |  p(x-z). \label{lemconvrob4}
\end{align}
From the uniform continuity of $G_{-}$, \eqref{lemconvrob3} is equal to zero. Now we analyse \eqref{lemconvrob4}.
In the finite-range case, $p(x-z)=0$ if $|x-z|$ is large enough. In the long-range case, since $p(x-z)=c_{\gamma} |x-z|^{-\gamma-1}$ and $\gamma > 2$, we get for every $\varepsilon >0$
\begin{align*}
&\lim_{n \rightarrow \infty}  \sum_{z=0}^{\infty} \sum_{x=-\infty}^{- \varepsilon n} \sup_{s \in [0,t]} |G_{-}(s,\tfrac{x}{n}) - G_{-}(s,0) |  p(x-z) \lesssim \lim_{n \rightarrow \infty} n^{1-\gamma} \int_0^{\infty} \int_{-\infty}^{-\varepsilon} (u-v)^{-1-\gamma} dv du =0.
\end{align*}
This shows that \eqref{lemconvrob4} is equal to zero. With an analogous reasoning, the same holds for  the rightmost term of \eqref{lemconvrob1}. This ends the proof.
\end{proof}
Now we present a result that was useful to treat \eqref{neuterm}.
\begin{prop} \label{lemconvneum}
Let $t\in [0,T]$  and $G \in \mcb S_{\textrm{Rob}}$. Then, 
\begin{equation} \label{lemneupos}\begin{split}
&\limsup_{\varepsilon \rightarrow 0^{+}} \limsup_{n \rightarrow \infty} \sup_{s \in [0,T]} \Big| \sum_{z=0}^{\varepsilon n -1} \sum_{x=0}^{\infty} \big[ n[G(s, \tfrac{x}{n}) - G(s, \tfrac{z}{n}) ] -  \partial_u G (s,0^{+})(x-z) \big]  p(x-z)  \Big|  =0,\\
&\limsup_{\varepsilon \rightarrow 0^{+}} \limsup_{n \rightarrow \infty}\sup_{s \in [0,T]} \Big|  \sum_{z=-\varepsilon n +1}^{-1} \sum_{x=-\infty}^{-1} \big[ n[G(s,\tfrac{x}{n}) - G(s,\tfrac{z}{n}) ] - \partial_u G(s,0^{-})(x-z) \big]  p(x-z)  \ \Big| =0.\end{split}
\end{equation}
\end{prop}
\begin{proof}
We prove only the first identity in  \eqref{lemneupos}, but we observe that the proof of second one is analogous. Since $G \in \mcb S_{\textit{Rob}}$, there exists $G_{+} \in \mcb S_{\textit{Dif}}$ such that $G(s,u) = G_{+}(s,u), \forall (s,u) \in [0,T] \times [0, \infty)$. Then for every $x,z \in \mathbb{N}$, there exists $\xi_{s,x,z}$ between $\frac{x}{n}$ and $\frac{z}{n}$ such that
\begin{align*}
n[G(s,\tfrac{x}{n}) - G(s,\tfrac{z}{n}) ] = n[G_{+}(s,\tfrac{x}{n}) - G_{+}(s,\tfrac{z}{n}) ] = (x-z) \partial_u G_{+}(s, \xi_{s,x,z}), \forall s \in [0,t].
\end{align*}
Since $0 \leq \eta_s^n(z) \leq 1$, the limit in the first identity in  \eqref{lemneupos} is bounded from above by
\begin{align}
&  \limsup_{\varepsilon \rightarrow 0^{+}} \limsup_{n \rightarrow \infty}  \sum_{z=0}^{\varepsilon n -1} \sum_{x=0}^{2 \varepsilon n -1} | \partial_u G_{+}(s,\xi_{s,x,z}) - \partial_u G_{+}(s,0) | |x-z|  p(x-z) \label{lemconvneu1} \\
+&  \limsup_{\varepsilon \rightarrow 0^{+}} \limsup_{n \rightarrow \infty} \sum_{z=0}^{\varepsilon n -1} \sum_{x=2 \varepsilon n}^{\infty} | \partial_u G_{+}(s,\xi_{s,x,z}) - \partial_u G_{+}(s,0) | (x-z)  p(x-z)\label{lemconvneu2}  . 
\end{align}
Since $\partial_u G_{+}$ is uniformly continuous then \eqref{lemconvneu1} is equal to zero. Now we analyse \eqref{lemconvneu2}.
In the finite-range case, $p(x-z)=0$ if $|x-z|$ is large enough. In the long-range case, since $p(x-z)=c_{\gamma} |x-z|^{-\gamma-1}$ and $\gamma > 2$, we get for every $\varepsilon >0$
\begin{align*}
&\lim_{n \rightarrow \infty}  \sum_{z=0}^{\varepsilon n -1} \sum_{x=2 \varepsilon n}^{\infty} | \partial_u G_{+}(s,\xi_{s,x,z}) - \partial_u G_{+}(s,0) | (x-z)  p(x-z) \lesssim \lim_{n \rightarrow \infty} n^{2-\gamma} \int_0^{\varepsilon} \int_{2 \varepsilon}^{\infty} (v-u)^{-\gamma} dv du =0.
\end{align*}
\end{proof}
Now we present a second result that was useful to treat \eqref{neuterm}.
\begin{prop} 
Let $t \in [0,T]$, $G \in \mcb S_{\textrm{Rob}}$. Then, 
\begin{equation} \label{princneupos}
\limsup_{\varepsilon \rightarrow 0^{+}} \limsup_{n \rightarrow \infty}  \sup_{s \in [0,T]} \Big|  \Big\{ \sum_{z=\varepsilon n}^{\infty} \sum_{x=0}^{\infty}  n[G(s, \tfrac{x}{n}) - G(s, \tfrac{z}{n}) ]   p(x-z)  - \frac{\sigma^2}{2} \frac{1}{n} \sum_{z=0}^{\infty} \Delta  G (s, \tfrac{z}{n})   \Big\}  \Big| =0
\end{equation}
and
\begin{equation} \label{princneuneg}
\limsup_{\varepsilon \rightarrow 0^{+}} \limsup_{n \rightarrow \infty} \sup_{s \in [0,T]} \Big|  \Big\{ \sum_{z=- \infty }^{- \varepsilon n} \sum_{x=-\infty}^{-1}  n[G(s, \tfrac{x}{n}) - G(s, \tfrac{z}{n}) ]   p(x-z)   - \frac{\sigma^2}{2} \frac{1}{n} \sum_{z=-\infty}^{-1} \Delta G (s, \tfrac{z}{n})  \Big\} \Big|  =0.
\end{equation}
\end{prop}
\begin{proof}
We will prove only \eqref{princneupos}, but we observe that the proof of \eqref{princneuneg} is analogous. Since $G \in \mcb S_{\textit{Rob}}$, there exists $G_{+} \in \mcb S_{\textit{Dif}}$ such that $G(s,u) = G_{+}(s,u), \forall (s,u) \in [0,T] \times [0, \infty)$. 
With the same reasoning we did in order to prove that the expression in \eqref{lemconvrob4} is zero, we can conclude that when $n\to+\infty$ and $\epsilon\to0$ the next two terms vanish:
\begin{align}
& \sup_{s \in [0,T]} \Big| \sum_{z=\varepsilon n}^{2 b_G n} \sum_{x=0}^{z- \varepsilon n}  n[G_{+}(s,\tfrac{x}{n}) - G_{+}(s,\tfrac{z}{n}) ]   p(x-z)   + \sum_{z=\varepsilon n}^{b_G n} \sum_{x=z+ \varepsilon n}^{ \infty}  n[G_{+}(s,\tfrac{x}{n}) - G_{+}(s,\tfrac{z}{n}) ]   p(x-z)  \ \Big|   \label{limprincneu1},\\
& \sup_{s \in [0,T]} \Big| \sum_{z=2 b_G n}^{\infty} \sum_{x=0}^{\infty}  n[G_{+}(s,\tfrac{x}{n}) - G_{+}(s,\tfrac{z}{n}) ]   p(x-z) \Big|. \label{limprincneu15}  
\end{align}
In the last line we used the fact that $G(s,u)=0$ if $|u| \geq b_G$. Since $\limsup_{n \rightarrow \infty} \sum_{ |r| \geq \varepsilon n} r^2 p(r) =0$ for every $\varepsilon >0$, it holds
\begin{equation} \label{limprincneu2}
\limsup_{\varepsilon \rightarrow 0^{+}} \limsup_{n \rightarrow \infty} \sup_{s \in [0,T]} \Big|  \frac{1}{n} \sum_{z=0}^{b_G n} \Delta G_{+} (s, \tfrac{z}{n})    \sum_{ |r| \geq \varepsilon n} \frac{r^2 p(r)}{2} \Big|  =   0.
\end{equation}
From a Taylor expansion on $G$ and the fact that $p$ is symmetric we can conclude that 
\begin{align}
  \limsup_{\varepsilon \rightarrow 0^{+}} \limsup_{n \rightarrow \infty}   \sup_{s \in [0,T]} \Big| \sum_{z=\varepsilon n}^{ 2 b_G n} \sum_{x=z-\varepsilon n +1}^{z + \varepsilon n -1}  n[G_{+}(s, \tfrac{x}{n}) - G_{+}(s, \tfrac{z}{n}) ]   p(x-z) 
  - \frac{1}{n} \sum_{z=0}^{2 b_G n}  \Delta G_{+} (s, \tfrac{z}{n})   \sum_{r=- \varepsilon n +1}^{\varepsilon n -1} \frac{r^2 p(r)}{2} \Big|=0. \label{limprincneu3} 
\end{align}
Above we used that $ \Delta G_{+}$ is uniformly continuous.
Since the limit in \eqref{princneupos} can be bounded from above by the sum of the limits of \eqref{limprincneu1}, \eqref{limprincneu15}, \eqref{limprincneu2} and \eqref{limprincneu3},  the proof ends.
\end{proof}
The next result is a immediate consequence of the last one.
\begin{cor} \label{princneum}
Let $t \in [0,T]$, $G \in \mcb S_{\textrm{Rob}}$. Then, 
\begin{equation*} 
\begin{split}
&\limsup_{\varepsilon \rightarrow 0^{+}} \limsup_{n \rightarrow \infty} \mathbb{E}_{\mu_n} \Big[ \Big| \int_{0}^{t} \Big\{ \sum_{z=\varepsilon n}^{\infty} \sum_{x=0}^{\infty}  n[G(s, \tfrac{x}{n}) - G(s, \tfrac{z}{n}) ]   p(x-z)  \eta_s^n(z) - \frac{\sigma^2}{2n}  \sum_{z=0}^{\infty} \Delta  G (s, \tfrac{z}{n})  \eta_s^n(z)  \Big\} ds \Big| \Big] =0,
\\
&\limsup_{\varepsilon \rightarrow 0^{+}} \limsup_{n \rightarrow \infty} \mathbb{E}_{\mu_n} \Big[ \Big| \int_{0}^{t} \Big\{ \sum_{z=- \infty }^{- \varepsilon n} \sum_{x=-\infty}^{-1}  n[G(s, \tfrac{x}{n}) - G(s, \tfrac{z}{n}) ]   p(x-z)  \eta_s^n(z) - \frac{\sigma^2}{2n}  \sum_{z=-\infty}^{-1} \Delta G (s, \tfrac{z}{n})  \eta_s^n(z)  \Big\} ds \Big| \Big] =0.\end{split}
\end{equation*}
\end{cor}

Finally, we present a result that was useful in Section \ref{sectight}.
\begin{prop} \label{tight2condaux}
For $G \in \mcb S_{ Dif}$,  it holds
\begin{align*}
\sum_{w,z }     \big[G\left(s,\tfrac{z}{n}\right)-G\left(s,\tfrac{w}{n}\right)\big ]^2 p(z-w)  \lesssim n^{-1} , \forall s \in [0,T].
\end{align*}
\end{prop}
\begin{proof}
From a Taylor expansion of first order, the leftmost term in last display can be estimated by 
\begin{align*}
 \frac{\| \partial_u G \|_{\infty}^2}{n}  \frac{2}{n}  \sum_{|x| \leq b_G n}  \sum_{r  } r^2 p(r)  = \frac{\| \partial_u G \|_{\infty}^2}{n}  \frac{2 \sigma^2 (2b_Gn +1)}{n} \leq \frac{6b_G \sigma^2 \|\partial_u G\|_{\infty}^2}{n},
\end{align*}
which gives the desired result.
\end{proof}

\section{Analysis tools} \label{secuniq}

In this section we will prove the uniqueness of the weak solutions of \eqref{eqhyddifrob}. Since we did not find in the literature the proof of uniqueness of our weak solutions we decided to prove it here. 

\subsection{Sobolev space results}

We rewrite Theorem 8.2 of \cite{brezis2010functional} for our convenience.
\begin{prop} \label{repcont}
\textcolor{black}{Let $f \in \mcb{H}^1(I)$, where $I$ is a one-dimensional interval.} If $\bar{I}$ denotes the closure of $I$, there exists one function $\tilde{f} \in C^0(\bar{I}) $ such that $f = \tilde{f}$ almost everywhere on $I$ and
\begin{align*}
\tilde{f}(y) - \tilde{f}(x) = \int_{x}^y \frac{df}{du}(u) du, \forall x,y \in \bar{I}. 
\end{align*}
\end{prop} 
This means that every function $f \in \mcb{H}^1(I)$ admits one \textit{continuous representative}  on $\bar{I}$, which will be denoted by $\tilde{f}$. Another very useful result is the integration by parts formula for $\mcb{H}^1(I)$, which comes from Corollary 8.10 in \cite{brezis2010functional}.
\begin{prop}
Let $f,g \in \mcb{H}^1(I)$. Then $f g \in \mcb{H}^1(I)$ and $\frac{d}{du} (f g) =\frac{df}{du}  g + f \frac{dg}{du}$. Furthermore, the formula for integration by parts holds
\begin{align*}
\int_{x}^{y}  \frac{df}{d u}(u) g(u) du = \tilde{f}(y) \tilde{g}(y) - \tilde{f}(x) \tilde{g}(x) - \int_{x}^{y} f(u) \frac{dg}{d u}(u)  du, \forall x,y \in \bar{I}.
\end{align*}
Above, $\tilde{f}$ and $\tilde{g}$ are the continuous representatives of $f$ and $g$, respectively.
\end{prop}
The next result is useful to prove that there is no macroscopic blockage of mass between $\mathbb{R}_{-}^{*}$ and $\mathbb{R}_{+}$ for $\beta \in [0,1)$, even when $\sigma_{\mcb S}^2 = \sigma ^2$.
\begin{prop} \label{H1reta}
Assume $f \in \mcb{H}^1(\mathbb{R} ^{*})$ and $f(0^{-})=f(0^{+})$. Then $f \in \mcb{H}^1(\mathbb{R})$.
\end{prop}
\begin{proof}
Let $\phi \in C_c^{\infty}(\mathbb{R})$. Denote $f_{-}: = f|_{\mathbb{R}_{-}^{*}}$, $f_{+}: = f|_{\mathbb{R}_{+}^{*}}$. Denote $\bar{f}_{-}$ and $\bar{f}_{+}$ the continuous representatives of $f_{-}$ and $f_{+}$ in $( -\infty, 0]$ and $[0, \infty)$, respectively. Then  $ \bar{f}_{-}(0) = f(0^{-})= f(0^{+})  = \bar{f}_{+}(0)$. The integration by parts formula in $\mcb{H}^1(\mathbb{R}_{-}^{*})$ leads to
\begin{equation} \label{HRneg}
\int_{\mathbb{R}_{-}} f(u) \phi'(u) du = \int_{\mathbb{R}_{-}} \bar{f}_{-}(u) \phi'(u) du = \bar{f}_{-}(0) \phi(0) - \int_{\mathbb{R}_{-}} \partial_u \bar{f}_{-}(u) \phi(u) du.
\end{equation}
In the same way, the integration by parts formula in $\mcb{H}^1(\mathbb{R}_{+}^{*})$ leads to
\begin{equation} \label{HRpos}
\int_{\mathbb{R}_{+}} f(u) \phi'(u) du = \int_{\mathbb{R}_{+}} \bar{f}_{+}(u) \phi'(u) du =- \bar{f}_{+}(0) \phi(0) - \int_{\mathbb{R}_{+}} \partial_u \bar{f}_{+}(u) \phi(u) du.
\end{equation}
Summing \eqref{HRneg} and \eqref{HRpos}, we get
\begin{align*}
 \int_{\mathbb{R}} f(u) \phi'(u) du  =& -\int_{\mathbb{R}_{-}} \partial_u \bar{f}_{-}(u) \phi(u) du -  \int_{\mathbb{R}_{+}} \partial_u \bar{f}_{+}(u) \phi(u) du = - \int_{\mathbb{R}} g(u) \phi(u) du,
\end{align*}
where for $u\in\bb R$, $g(u) = \partial_u \bar{f}_{-}(u)\mathbbm{1}_{u<0}+ \partial_u \bar{f}_{+}(u)\mathbbm{1}_{u >0}$. Since $f, g \in L^2(\mathbb{R})$, then $f \in \mcb{H}^1(\mathbb{R})$. 
\end{proof}
The following result can be proved using the same ideas as in Lemma 8.2 of \cite{brezis2010functional}.
\begin{prop} \label{rhoposH1}
Let $f \in \mcb {H}^1(\mathbb{R}_{-}^{*}), g \in \mcb{H}^1(\mathbb{R}_{+}^{*})$. Let $\tilde{f}_{-}$ and $\tilde{g}_{+}$ be the even extensions of the continuous representatives $\tilde{f}$ and $\tilde{g}$, respectively. Then $\tilde{f}_{-}, \tilde{g}_{+} \in  \mcb{H}^1(\mathbb{R})$.
\end{prop}

 From Theorem 8.7 in \cite{brezis2010functional}, $C_c^{\infty}(\mathbb{R})$ is dense in $\mcb{H}^1(\mathbb{R})$ with the norm $|| \cdot ||_{\mcb{H}^1(\mathbb{R})}$. From Proposition 23.2 (d) in \cite{zeidler1989nonlinear}, we have that $P \big([0,T], \mcb{H}^1(\mathbb{R})\big)$ is dense in $L^2\big(0,T; \mcb{H}^1(\mathbb{R}) \big)$ with the norm $|| \cdot ||_{L^2\big(0,T; \mcb{H}^1(\mathbb{R}) \big)} $. Combining both results, we get that 
$\mcb S_{\textrm{Dif}}$ is dense in $L^2\big(0,T; \mcb{H}^1(\mathbb{R}) \big)$. This allows  enunciating the following lemma, whose proof can be adapted from the proof of Lemma A.1 in \cite{bernardin2017slow}. 

\begin{lem} \label{lemuniqdif}
Let $\varrho \in L^2 \left( 0,T; \mcb{H}^1(\mathbb{R}) \right)$ and $(H_k)_{k \geq 1}$ be a sequence of functions in $\mcb S_{\textrm{Dif}}$ converging to $\varrho$ with respect to the norm of $L^2 \left( 0,T; \mcb{H}^1(\mathbb{R}) \right)$. We define $G_k \in \mcb S_{\textrm{Dif}}$ by 
\begin{align*}
G_k (t,u) = \int_t^T H_k (s,u) ds, \forall t \in [0,T], \forall u \in \mathbb{R}, \forall k \geq 1.
\end{align*}
Then for every $I \subset \mathbb{R}$, it holds
\begin{equation} \label{lemuniq1}
\lim_{k \rightarrow \infty} \int_0^T \int_{I} \varrho(s,u) \partial_s G_k (s,u) du ds = - \int_0^T \int_{I} [ \varrho(s,u) ]^2 duds,
\end{equation}
and
\begin{equation} \label{lemuniq2}
\lim_{k \rightarrow \infty}  \int_0^T \int_{I} \partial_u \varrho(s,u) \partial_u G_k(s,u) du ds  = \frac{1}{2} \int_{I} \Big[ \frac{\partial}{\partial u} \Big(   \int_0^T  \varrho(s,u) ds \Big) \Big]^2 du.
\end{equation}
\end{lem}

In order to prove the uniqueness of weak solutions of \eqref{eqhyddifrob}, the following lemma will be useful.
\begin{lem} \label{lemuniqrob}
Assume $f:[0,T] \rightarrow \mathbb{R}$, $\varrho:[0,T] \times \mathbb{R}$ are bounded, $\varrho \in  L^2 \big( 0,T; \mcb{H}^1 (\mathbb{R}) \big)$ and $\varrho(s, \cdot) \in C^0(\mathbb{R})$, for almost every  $s \in [0,T]$. Let $(H_{k})_{k \geq 1}$  be a sequence in $\mcb S_{\textrm{Dif}}$  converging to $\varrho$ with respect to the norm of $L^2 \left( 0,T; \mcb{H}^1 (\mathbb{R})\right)$. Then
\begin{align*}
\lim_{k \rightarrow \infty}  \int_0^T \int_s^T \varrho(s,0) H_k(r,0) dr ds  = \frac{1}{2} \Big[ \int_0^T \varrho(s,0)  \Big]^2.
\end{align*}  
\end{lem}
\begin{proof}
This lemma is strongly inspired in Section 4.4. of \cite{bernardin2021hydrodynamic}. Holder's inequality leads to 
\begin{align*}
\Big| \int_0^T \int_s^T \varrho(s,0)& H_k(r,0) dr ds - \frac{1}{2} \Big[ \int_0^T \varrho(s,0)  \Big]^2 \Big|= \Big|  \int_0^T \int_s^T \varrho(s,0) [ H_k(r,0) - \varrho(r,0)] dr ds \Big| \\
\leq& T^{\frac{3}{2}} \| \varrho \|_{\infty} \sqrt{  \int_0^T  [ H_k(r,0) - \varrho(r,0)]^2 dr} =  T^{\frac{3}{2}} \| \varrho \|_{\infty} \sqrt{  \int_0^T  [ f_k(r,0)]^2 dr}, 
\end{align*}
 where $f_k:=H_k-\varrho, \forall k \geq 1$.  Since $\varrho(r, \cdot) \in C^0(\mathbb{R})$, for almost every  $r \in [0,T]$, from Theorem 9.12 of \cite{brezis2010functional}, there exists a constant $C$ independent of $k$ such that 
\begin{align*}
|f_k(r,0)|  \leq \int_{0}^{1}  [|f_k(r,0)-f_k(r,u)|  + |f_k(r,u)| ] du  
\leq C \|f_k(r, \cdot) \|_{\mcb{H}^1(\mathbb{R})} + \int_{0}^{1} |f_k(r,u)| du \leq( C +1) \|f_k(r, \cdot) \|_{\mcb{H}^1(\mathbb{R})}.
\end{align*}
Integrating over time and since $(f_k)_{k \geq 1}$ converges to zero in $L^2 \big( 0,T; \mcb{H}^1 (\mathbb{R}) \big)$, the proof ends. 

\end{proof}
 
\subsection{Uniqueness of weak solutions}

We observe that weak solutions of \eqref{eqhyddifrob} deal with $\mcb S_{\textit{Rob}}$ as the  space of test functions. 
The uniqueness of the weak solutions of \eqref{eqhyddifrob}  is equivalent to the following result.
\begin{prop} \label{uniqeqhyddifrobneu}
Let $\varrho_1, \varrho_2$ be such that $\varrho_1 -a, \varrho_2 -a  \in L^2 \big( 0,T; \mcb{H}^1 (\mathbb{R}^{*}) \big)$, for some $a \in (0,1)$. If 
\begin{align*}
F_{\textrm{Rob}}(t, \varrho_1, G,  \mcb g,\kappa) = 0 = F_{\textrm{Rob}}(t, \varrho_2, G, \mcb g,\kappa), \forall t \in [0,T], \forall G \in \mcb S_{\textrm{Rob}},
\end{align*} 
then $\varrho_1 = \varrho_2$ almost everywhere in $[0,T] \times \mathbb{R}$.
\end{prop}
\begin{proof}
Denote $\varrho_3:= \varrho_1 - \varrho_2= [\varrho_1 - a] - [\varrho_2 - a]$. Then $\varrho_3 \in  L^2 \big( 0,T; \mcb{H}^1 (\mathbb{R}^{*}) \big)$ and $\varrho_3(s,u) \in \mcb{H}^1(\mathbb{R} ^{*})$, for almost every  $ s \in [0,T]$. Let $\tilde{\varrho}_{3,-}$ be the even extension of the continuous representative of $\varrho_3(s,\cdot)|_{\mathbb{R}_{-}^{*}}$ and define $\tilde{\varrho}_{3,+}$ in the same way, replacing $\mathbb{R}_{-}^{*}$ by $\mathbb{R}_{+}^{*}$. From Proposition \ref{rhoposH1}, it follows that $\varrho_{3,-}(s, \cdot)$ and $\varrho_{3,+}(s, \cdot)$ are in $ \mcb{H}^1(\mathbb{R})$. Then for every $t \in [0,T]$ and  for every $G \in \mcb S_{\textit{Rob}}$ we get that
\begin{align*}
0= & \int_{\mathbb{R}} \varrho_3(t,u) G(t,u) du -  \int_0^t \int_{\mathbb{R}_{-}} \tilde{\varrho}_{3,-}(s,u)  \partial_s  G(s,u) du ds -  \int_0^t \int_{\mathbb{R}_{+}} \tilde{\varrho}_{3,+}(s,u)  \partial_s  G(s,u) du ds  \\
+& \frac{\sigma^2}{2}  \int_0^t \Big[  \partial_u G(s,0^{-}) \tilde{ \varrho}_{3,-}(s,0)  - \int_{\mathbb{R}_{-}}  \tilde{\varrho}_{3,-}(s,u)  \Delta G(s,u) du  \Big] ds \\
+& \frac{\sigma^2}{2}  \int_0^t \Big[  - \partial_u G(s,0^{+}) \tilde{ \varrho}_{3+}(s,0)  - \int_{\mathbb{R}_{+}}  \tilde{\varrho}_{3,+}(s,u)  \Delta G(s,u) du  \Big] ds \\
+ &  \frac{\kappa \sigma^2}{2} \int_0^t  [    \tilde{ \varrho}_{3+}(s,0)  -  \tilde{ \varrho}_{3,-}(s,0) ]  [  G(s,0^{+})    - G(s,0^{-})  ] ds.
\end{align*}
 Applying the integration by parts formula for $\mcb{H}^1 (\mathbb{R}_{-}^{*})$ and $\mcb{H}^1 (\mathbb{R}_{+}^{*})$, we get, for every $t \in [0,T]$ and for every $G \in \mcb {S}_{\textit{Rob}}$, that
\begin{align*}
0= &  \int_{\mathbb{R}_{-}} \tilde{\varrho}_{3,-}(t,u) G(t,u) du - \int_0^t \int_{\mathbb{R}_{-}} \tilde{\varrho}_{3,-}(s,u)  \partial_s  G(s,u) du ds + \frac{\sigma^2}{2}  \int_0^t  \int_{\mathbb{R}_{-}} \partial_u  \tilde{\varrho}_{3,-}(s,u)  \partial_u G(s,u) du   ds \\
+ & \int_{\mathbb{R}_{+}} \tilde{\varrho}_{3,+}(t,u) G(t,u) du - \int_0^t \int_{\mathbb{R}_{+}} \tilde{\varrho}_{3,+}(s,u)  \partial_s  G(s,u) du ds + \frac{\sigma^2}{2}  \int_0^t  \int_{\mathbb{R}_{+}} \partial_u  \tilde{\varrho}_{3,+}(s,u)  \partial_u G(s,u) du   ds \\
+& \frac{\kappa \sigma^2}{2} \int_0^t  [    \tilde{ \varrho}_{3+}(s,0)  -  \tilde{ \varrho}_{3,-}(s,0) ]  [  G(s,0^{+})    - G(s,0^{-})  ] ds,
\end{align*}
 Since $\tilde{\varrho}_{3,-}, \tilde{\varrho}_{3,+}  \in  L^2 \big( 0,T; \mcb{H}^1 (\mathbb{R}) \big)$, there exist $(H_{k,-})_{k \geq 1}, (H_{k,+})_{k \geq 1}$ in $\mcb {S}_{\textit{Dif}}$ such that $(H_{k,-})_{k \geq 1}$ (resp. $(H_{k,+})_{k \geq 1}$)  converges to $\tilde{\varrho}_{3,-}$ (resp.  $\tilde{\varrho}_{3,+}$) with respect to the norm of $L^2 \left( 0,T; \mcb{H}^1 (\mathbb{R})\right)$. Define $G_{k,-} (t,u) :=  \int_t^T H_{k,-} (s,u) ds, \forall (t,u) \in [0,T] \times \mathbb{R}, \forall k \geq 1$ and $G_{k,+} (t,u) :=  \int_t^T H_{k,+} (s,u) ds, \forall (t,u) \in [0,T] \times \mathbb{R}, \forall k \geq 1$. Moreover, define $G_{k} \in \mcb {S}_{\textit{Rob}}$ by
\begin{align*}
G_{k} (t,u) =\mathbbm{1}_{u\in(-\infty,0)}G_{k,-} (t,u)+\mathbbm{1}_{u\in [0, \infty)} G_{k,+} (t,u), \forall (t,u) \in [0,T] \times \mathbb{R}, \forall k \geq 1.
\end{align*}
In particular, $G_{k}  (T,u) = 0, \forall u \in \mathbb{R}, \forall k \geq 1$. Taking $t=T$  and $G=G_{k}$, we get
\begin{align}
0= &   - \int_0^T \int_{\mathbb{R}_{-}} \tilde{\varrho}_{3,-}(s,u)  \partial_s  G_{k,-}(s,u) du ds + \frac{\sigma^2}{2}  \int_0^T  \int_{\mathbb{R}_{-}} \partial_u  \tilde{\varrho}_{3,-}(s,u)  \partial_u G_{k,-}(s,u) du   ds  \nonumber \\
- & \int_0^T \int_{\mathbb{R}_{+}} \tilde{\varrho}_{3,+}(s,u)  \partial_s  G_{k,+}(s,u) du ds + \frac{\sigma^2}{2}  \int_0^T  \int_{\mathbb{R}_{+}} \partial_u  \tilde{\varrho}_{3,+}(s,u)  \partial_u G_{k,+}(s,u) du   ds \nonumber \\
+& \frac{\kappa \sigma^2}{2} \int_0^T  \int_s^T  [    \tilde{ \varrho}_{3+}(s,0)  -  \tilde{ \varrho}_{3,-}(s,0) ]  [  H_{k,+}(r,0)    - H_{k,-}(r,0)  ] dr ds, \forall k \geq 1. \label{equniqbarfor}
\end{align}
Since $G_{k,-}$, $G_{k,+}$ and $H_{k,+}-H_{k,-}$ are in $\mcb {S}_{\textit{Dif}}$, we can  use Lemma \ref{lemuniqdif} and Lemma \ref{lemuniqrob}. Taking the limit in \eqref{equniqbarfor} when $k \rightarrow \infty$, we get
\begin{align*}
& \int_0^T \int_{\mathbb{R}_{-}} [ \tilde{\varrho}_{3,-}(s,u) ]^2 du ds + \int_0^T \int_{\mathbb{R}_{+}} [ \tilde{\varrho}_{3,+}(s,u) ]^2 du ds 
+  \frac{\sigma^2}{4} \int_{\mathbb{R}_{-}} \Big[ \frac{\partial}{\partial u} \Big(   \int_0^T  \tilde{\varrho}_{3,-}(s,u) ds \Big) \Big]^2 du \\&+  \frac{\sigma^2}{4} \int_{\mathbb{R}_{+}} \Big[ \frac{\partial}{\partial u} \Big(   \int_0^T  \tilde{\varrho}_{3,+}(s,u) ds \Big) \Big]^2 du 
+ \frac{\kappa \sigma^2}{4}  \Big( \int_0^T  [    \tilde{ \varrho}_{3+}(s,0)  -  \tilde{ \varrho}_{3,-}(s,0) ] ds \Big)^2=0,
\end{align*}
which implies that $\tilde{\varrho}_{3,-}, \varrho_{3,-}, \varrho_3$ are equal to zero almost everywhere on $[0,T] \times \mathbb{R}_{-}$ and $\tilde{\varrho}_{3,+}, \varrho_{3,+},  \varrho_3$ are equal to zero almost everywhere on $[0,T] \times \mathbb{R}_{+}$. Then $\varrho_1 = \varrho_2$ almost everywhere on $[0,T] \times \mathbb{R}$. 
\end{proof}
The proof for the uniqueness of weak solutions of \eqref{eqhyddifreal} is analogous to the proof given above, so that we omit details. 

\quad

\thanks{ {\bf{Acknowledgements: }}
P.C. thanks FCT/Portugal for support through the project Lisbon Mathematics PhD (LisMath). P.C. and P.G. thank  FCT/Portugal for financial support
through CAMGSD, IST-ID, projects UIDB/04459/2020 and UIDP/04459/2020. B.J.O thanks  Universidad Nacional de Costa Rica  for sponsoring the participation in  this article. This project has received funding from the European Research Council (ERC) under  the European Union's Horizon 2020 research and innovative programme (grant agreement   n. 715734).}

\bibliographystyle{plain}
\bibliography{bibliografia}

\begin{thebibliography}{10}

\bibitem{baldasso}
Rangel Baldasso, Ot\'{a}vio Menezes, Adriana Neumann, and Rafael~R. Souza.
\newblock Exclusion process with slow boundary.
\newblock {\em J. Stat. Phys.}, 167(5):1112--1142, 2017.

\bibitem{bernardinARMA}
Cédric Bernardin, Patr\'{\i}cia Gon\c{c}alves, and Byron Jim\'{e}nez-Oviedo.
\newblock A microscopic model for a one parameter class of fractional
  {L}aplacians with {D}irichlet boundary conditions.
\newblock {\em Arch. Ration. Mech. Anal.}, 239(1):1--48, 2021.

\bibitem{bernardin2021hydrodynamic}
Cedric Bernardin, Pedro Cardoso, Patr\'{\i}cia Gon\c{c}alves, and Stefano
  Scotta.
\newblock Hydrodynamic limit for a boundary driven super-diffusive symmetric
  exclusion.
\newblock {\em arXiv preprint arXiv:2007.01621}, 2021.

\bibitem{bernardin2017slow}
C\'{e}dric Bernardin, Patr\'{\i}cia Gon\c{c}alves, and Byron
  Jim\'{e}nez-Oviedo.
\newblock Slow to fast infinitely extended reservoirs for the symmetric
  exclusion process with long jumps.
\newblock {\em Markov Process. Related Fields}, 25(2):217--274, 2019.

\bibitem{brezis2010functional}
Haim Brezis.
\newblock {\em Functional analysis, {S}obolev spaces and partial differential
  equations}.
\newblock Universitext. Springer, New York, 2011.

\bibitem{CGJ2}
Pedro Cardoso, Patricia Gon{\c{c}}alves, and Byron Jiménez-Oviedo.
\newblock Hydrodynamics of super-diffusive long-range symmetric exclusion with
  a slow barrier.
\newblock {\em in preparation}, 2021+.

\bibitem{FGNAIHP}
Tertuliano Franco, Patr\'{\i}cia Gon\c{c}alves, and Adriana Neumann.
\newblock Hydrodynamical behavior of symmetric exclusion with slow bonds.
\newblock {\em Ann. Inst. Henri Poincar\'{e} Probab. Stat.}, 49(2):402--427,
  2013.

\bibitem{franco2015phase}
Tertuliano Franco, Patr\'{\i}cia Gon\c{c}alves, and Adriana Neumann.
\newblock Phase transition of a heat equation with {R}obin's boundary
  conditions and exclusion process.
\newblock {\em Trans. Amer. Math. Soc.}, 367(9):6131--6158, 2015.

\bibitem{tertumariana}
Tertuliano Franco and Mariana Tavares.
\newblock Hydrodynamic limit for the {SSEP} with a slow membrane.
\newblock {\em J. Stat. Phys.}, 175(2):233--268, 2019.

\bibitem{goncalvesscotta}
Patr\'{\i}cia Gon\c{c}alves and Stefano Scotta.
\newblock Diffusive to super-diffusive behavior in boundary driven exclusion.
\newblock {\em to appear in Markov Process. Related Fields}, 2021.

\bibitem{jara2009hydrodynamic}
Milton Jara.
\newblock Hydrodynamic limit of particle systems with long jumps.
\newblock {\em arXiv preprint arXiv:0805.1326}, 2008.

\bibitem{kipnis1998scaling}
Claude Kipnis and Claudio Landim.
\newblock {\em Scaling limits of interacting particle systems}, volume 320 of
  {\em Grundlehren der Mathematischen Wissenschaften [Fundamental Principles of
  Mathematical Sciences]}.
\newblock Springer-Verlag, Berlin, 1999.

\bibitem{sunder}
Sunder Sethuraman and Doron Shahar.
\newblock Hydrodynamic limits for long-range asymmetric interacting particle
  systems.
\newblock {\em Electron. J. Probab.}, 23:Paper No. 130, 54, 2018.

\bibitem{zeidler1989nonlinear}
Eberhard Zeidler.
\newblock {\em Nonlinear functional analysis and its applications. {II}/{A}}.
\newblock Springer-Verlag, New York, 1990.
\newblock Linear monotone operators, Translated from the German by the author
  and Leo F. Boron.

\end{thebibliography}

\Addresses

\end{document}